\documentclass[11pt, a4paper]{amsart}

\usepackage[table]{xcolor}
\usepackage{amssymb}
\usepackage{amsthm}
\usepackage{amsaddr}
\usepackage{subcaption}
\usepackage{tikz}
\usepackage{amsmath}
\usepackage{mathtools}
\usepackage{chngcntr}
\usepackage[utf8]{inputenc}
\usepackage[sort, nocompress]{cite}
\usepackage[margin=0.75in]{geometry}

\usepackage{fancyhdr}

\newtheorem{theorem}{Theorem}[section]
\newtheorem{proposition}[theorem]{Proposition}
\newtheorem{lemma}[theorem]{Lemma}
\newtheorem{corollary}[theorem]{Corollary}
\newtheorem{assumption}[theorem]{Assumption}
\newtheorem{remark}[theorem]{Remark}

\newtheorem{definition}[theorem]{Definition}
\newtheorem{example}[theorem]{Example}

\counterwithin{figure}{section}
\counterwithin{table}{section}
\counterwithin{equation}{section}
\counterwithin{theorem}{section}

\newcommand{\RefRem}[1]{}

\newcommand{\mc}{\mathcal}
\newcommand{\mb}{\mathbb}
\newcommand{\bs}{\boldsymbol}
\renewcommand{\H}{\bs{\mc{H}}}

\newcommand{\Lp}[3][2]{\ensuremath{L^{#1}\mathopen{}\left(#2,\,#3\right)\mathclose{}}}

\newcommand{\LL}[1]{\ensuremath{\mathcal{L}\mathopen{}\left(#1\right)\mathclose{}}}
\newcommand{\Lnorm}{\ensuremath{\mathcal{L}}}
\newcommand{\B}{\Lnorm}

\newcommand{\HSS}[1]{\ensuremath{\mathrm{HS}\mathopen{}\left(#1\right)\mathclose{}}}
\newcommand{\HSnorm}{\ensuremath{\mathrm{HS}}}
\newcommand{\Tr}[1]{\ensuremath{\mathrm{Tr}\mathopen{}\left(#1\right)\mathclose{}}}

\newcommand{\Trp}[1]{\ensuremath{\mathrm{Tr}^+\mathopen{}\left(#1\right)\mathclose{}}}
\newcommand{\State}[1]{\ensuremath{\mathcal{S}\mathopen{}\left(#1\right)\mathclose{}}}
\newcommand{\Trnorm}{\ensuremath{\mathrm{Tr}}}

\newcommand{\X}{\ensuremath{\boldsymbol{\mathcal{H}}}}
\newcommand{\uu}{\psi}

\newcommand{\R}{\ensuremath{\mathbb{R}}}
\newcommand{\C}{\ensuremath{\mathbb{C}}}
\newcommand{\N}{\ensuremath{\mathbb{N}}}
\newcommand{\Z}{\ensuremath{\mathbb{Z}}}

\newcommand{\norm}[2]{\ensuremath{\mathopen{}\left\|#1\right\|_{#2}\mathclose{}}}
\newcommand{\normsq}[2]{\ensuremath{\mathopen{}\left\|#1\right\|^2_{#2}\mathclose{}}}

\newcommand{\set}[1]{\ensuremath{\mathopen{}\left\{#1\right\}\mathclose{}}}
\newcommand{\seq}[1]{\ensuremath{\boldsymbol{#1}}}

\newcommand{\rankb}[2]{\ensuremath{\operatorname{rank}_{#2}\mathopen{}\left(#1\right)\mathclose{}}}

\newcommand{\supp}[1]{\ensuremath{\operatorname{supp}\mathopen{}\left(#1\right)\mathclose{}}}

\newcommand{\tr}[1]{\operatorname{tr}\mathopen{}\left[#1\right]\mathclose{}}
\newcommand{\ptr}[2]{\operatorname{tr}_{#1}\mathopen{}\left[#2\right]\mathclose{}}

\newcommand{\inp}[2]{\mathopen{}\left\langle #1,\, #2\right\rangle\mathclose{}}

\newcommand{\comm}[2]{\mathopen{}\left[#1,\, #2\right]\mathclose{}}

\newcommand{\adj}[1]{{#1}^*}

\newcommand{\id}{\mathbb{I}}
\newcommand{\OB}{\ensuremath{\tilde{O}_B}}
\newcommand{\diff}{\Theta}

\renewcommand{\d}{\mathop{}\,\mathrm{d}}
\newcommand{\card}[1]{\# #1}

\renewcommand{\exp}[1]{\operatorname{exp}\mathopen{}\left[#1\right]\mathclose{}}

\newcommand{\s}[2][]{S^{#1}\mathopen{}\left(#2\right)\mathclose{}}
\newcommand{\dom}[1]{\mathcal{D}\mathopen{}\left(#1\right)\mathclose{}}
\renewcommand{\to}[1]{\mathcal{T}\mathopen{}\left(#1\right)\mathclose{}}
\renewcommand{\oe}[3]{\to{\exp{\int_{#2}^{#3}#1(\tau)\d\tau}}}

\DeclareMathOperator*{\dist}{dist}

\newcommand{\alaw}{\mathrm{Area}}

\newcommand{\cut}{\mathrm{Cut}}
\newcommand{\E}{\ensuremath{\mathbb{E}}}

\renewcommand{\i}{\ensuremath{\mathrm{i}}}

\title{Low-Rank Approximability and Entropy Area Laws for
       Ground States of Unbounded Hamiltonians}
\author{Mazen Ali}

\address{Germany, Ulm University, Institute for Numerical Mathematics}
\email{mazen.ali@uni-ulm.de}

\keywords{Low-Rank Approximation, Entropy Area Laws, Nearest Neighbor Interaction,
Partial Differential Equations, Eigenfunctions, Ground State}
\subjclass[2010]{46N50 (primary), 41A30 (secondary)}

\begin{document}

\begin{abstract}
    We show how local bounded
    interactions in an unbounded Hamiltonian
    lead to eigenfunctions
    with favorable low-rank properties. To this end, we utilize
    ideas from quantum entanglement of multi-particle spin systems.
    We begin by analyzing the connection between entropy
    area laws and low-rank approximability. The characterization for
    1D chains such as Matrix Product States (MPS) / Tensor Trains (TT)
    is rather
    extensive though incomplete.
    We then show that a Nearest Neighbor Interaction (NNI) Hamiltonian
    has eigenfunctions that are approximately separable in a certain
    sense. Under a further assumption on the approximand, we show that this
    implies a constant entropy bound.

    To the best of our knowledge, this work is
    the first analysis of
    low-rank approximability for unbounded Hamiltonians. Moreover,
    it extends previous results on entanglement entropy area laws to
    unbounded operators.
    The assumptions include a variety of self-adjoint operators and have
    a physical interpretation.
    The weak points are the aforementioned
    assumption on the approximand and that the validity is limited to
    MPS/TT formats.
\end{abstract}

\maketitle

\section*{Acknowledgement}
I would like to thank Matthew Hastings for the very helpful discussions
about this work. Particularly for \eqref{eq:thanks1},
\eqref{eq:thanks2} and insights on the
current standing of area law research. I would also like to thank
Alexander Nüßeler for the coffee talks about physics and literature suggestions
that inspired this work.

\section{Introduction}
Can we represent or approximate a function $f:\R^d\rightarrow\C$ with a
complexity that does not grow exponentially in the dimension $d$? Assuming
no specific structure\footnote{By this we mean, e.g., assuming only
classical smoothness
for $f$ which is not sufficient for describing low-rank approximability.}
for $f$, the general answer to this question is
``No'' (see \cite{Schneider, Novak}). This is commonly known as the
\emph{curse of dimensionality}.
Nonetheless, the development of low-rank tensor methods in recent decades
has shown that this curse can be broken for a variety of models
(see \cite{BD, HB, BSU, TN}).
Can we systematically identify the necessary structures that lead to low-rank
approximability?

The (approximation) theory on
this topic is very scarce. We are aware of only one result in
\cite{DahmenTS},
where the authors considered how the inverse of
a Laplace-like operator preserves low-rank structure.
The original motivation for our work was a theory that describes the
structure necessary for low-rank
approximability. We consider a class of
unbounded operators
and show how local interactions in the operator structure lead to favorable
approximability of eigenfunctions w.r.t.\ to growing $d$.

Fortunately, we do not have to start from scratch. In physics the phenomenon
of quantum entanglement has been known since as early as 1935 (see
\cite{EPR}). The study
of multi-particle quantum systems has led to intriguing connections between
the holographic principle, entropy area laws and approximablity by Matrix
Product States (MPS), the latter being a particular kind
of a tensor format also known as the Tensor Train (TT) format.
The quantum theoretical approach to approximability offers an entire
set of powerful tools that we can use to answer the question posed in the
beginning of this introduction.
See \cite{PlenioAL} for an overview of quantum entanglement entropy
area laws.

This work intertwines three things: a review of the ideas from quantum
entanglement relevant to low-rank approximation; a rigorous formulation and
proof of some of these ideas; an area law for a class of Hamiltonians.
The main results
are an approximation estimate for the ground state projection
(Theorem \ref{thm:obolor})
and an area law for 1D continuous variable systems
(Theorem \ref{thm:alaw}). This proof illustrates
essential mathematical ingredients that connect local
interactions with low-rank approximability of eigenfunctions.

The proof is mainly based on ideas from \cite{Hastings},
where the author considers NNI
systems of finite bond dimension\footnote{Mathematically represented by
a matrix acting on a finite dimensional Hilbert space.}.
To the best
of our knowledge, this was the first proof in physics that
a 1D spin system
obeys an area law. Since then other
proofs for discrete variable systems appeared, with sharper bounds and
more general assumptions.

For instance, in \cite{ExpDecay} the authors considered
the more general setting of exponentially decaying correlations
and this bound was later improved in \cite{ExpDecayBetter}.
In \cite{LongRange} the authors considered long range interactions.
In \cite{AradSubExp}
the authors significantly improved the bound for NNI systems
w.r.t.\ the spectral gap.
This was also used in \cite{AndreNNI}, to show that
discretized
NNI operators with general right hand sides
lead to approximable solutions. In \cite{GChain, GLattice} the authors showed
area laws for Gaussian states in 1D and 2D, i.e.,
(continuous variable) harmonic oscillators.

We choose to follow the ideas from \cite{Hastings}, since we
believe these are the most natural to extend to the infinite-dimensional
setting of PDEs (continuous variable systems). Of course, we do not
claim that other approaches are not feasible.

The outline of this paper is as follows. In Section \ref{sec:notation},
we introduce the main notation and terminology used through out the work.
In Section \ref{sec:entropy}, we take a closer look at the connection between entropy
scaling, approximability and discuss the issue of entropy discontinuity.
In Section \ref{sec:alaw}, we introduce the (NNI) Hamiltonian
operator class we consider,
prove some properties of the ground state and conclude with an
entropy scaling bound. The latter will require an assumption on the
approximate ground state projection that requires further investigation.
In Section \ref{sec:summ}, we conclude by summarizing the key steps of the proof,
discuss the assumptions made, evaluate the potential and limitations of our
approach.

\section{Notation}\label{sec:notation}
In this work we consider a separable Hilbert space which we denote by
$\H$, with the corresponding inner product $\inp{\cdot}{\cdot}_{\H}$ and
norm $\norm{\cdot}{\H}$. We assume $\H$ is a topological
tensor product of order $d$,
\begin{align*}
    \H=\mc H_1\otimes\cdots\otimes \mc H_d,
\end{align*}
where $d$ is (a multiple of) the
number of particles in a corresponding model of a quantum system.

Note the distinction between the dimension of a tensor network, which we
denote by $D$, and the order of the tensor product, which we denote by $d$.
The dimension of the tensor network $D$ refers to the spatial dimension
of the graph representing the network. E.g., particles ordered in a chain
represent a 1D system. The corresponding tensor format is MPS or TT, which
is a 1D tensor network. If particles are ordered on a lattice, the
corresponding system is 2D and the corresponding tensor format is, e.g.,
Projected Entangled Pair States (PEPS).
In all these examples $d$ is a multiple of the number of particles and is
typically large. In this work we focus on 1D systems (see next section for
more details).

We assume that $\H$ is equipped with the canonical inner product
\begin{align}\label{eq:caninp}
    \inp{\psi_1\otimes\cdots\otimes\psi_d}{\phi_1\otimes\cdots\otimes\phi_d}_{\H}
    =\prod_{j=1}^d\inp{\psi_j}{\phi_j}_{\mc H_j},\quad\psi_j,\phi_j\in\mc H_j.
\end{align}
Moreover, for any tensor product space of the form
$\H_{\alpha}:=\bigotimes_{j\in\alpha}\mc H_j$,
for $\alpha\subset\{1,\ldots,d\}$, we assume $\H_\alpha$ is equipped with the
canonical inner product as well. This implies many nice properties for
$\H_\alpha$ and tensor product operators on $\H_\alpha$ that we
can take for granted. In particular,
$\|\cdot\|_{\H_\alpha}$ is a uniform crossnorm. See \cite[Chapter 4]{HB}
for more details or \cite[Chapter 6.7]{HB} and
\cite{ANSVD1, ANSVD2} for the case where this is not satisfied.

We consider linear operators $T:\dom{T}\rightarrow\H$, where $\dom{T}$
is some subspace of $\H$ (typically assumed to be dense in $\H$).
We use $\LL{\H}$ to denote the space of all bounded operators. Note that
w.l.o.g.\ we can take
$\dom{T}=\H$ if $\dom{T}$ is dense in $\H$ and $T$ is bounded.
The operator norm is
\begin{align*}
    \|T\|_{\B}=\sup_{\psi\in\H\setminus\{0\}}
    \frac{\|T\psi\|_{\H}}{\|\psi\|_{\H}}.
\end{align*}
The Hilbert adjoint is denoted by $T^*$.
For self-adjoint operators we can define a partial ordering via
\begin{align*}
    T\geq 0\quad \Leftrightarrow\quad\inp{\psi}{T\psi}_{\H}\geq 0,
\end{align*}
for all
$\psi\in\dom{T}$. We refer to $T$ as being positive in this case.
Consequently
\begin{align*}
    T_1\geq T_2\quad \Leftrightarrow\quad T_1-T_2\geq 0.
\end{align*}
Note that positivity already implies we assume self-adjointess.

For a complete orthonormal system $\{e_k\}_{k\in\N}\subset\H$ and
$T\in\LL{\H}$, the \emph{trace} is defined as
\begin{align*}
    \tr{T}:=\sum_{k=1}^\infty \inp{e_k}{Te_k}_{\H}.
\end{align*}
Let $|T|:=(T^*T)^{1/2}$ denote the absolute value of $T$. If $\tr{|T|}<\infty$,
then we say $T$ is trace class in which case $\tr{T}$ is well defined and
independent of the choice of $\{e_k\}_{k\in\N}$. We denote the space of
trace class operators by $\Tr{\H}$ with the corresponding
norm $\|T\|_{\Trnorm}:=\tr{|T|}$. Since $\H$ is a Hilbert space, trace class
operators coincide with nuclear operators. We also use the following notation
\begin{align*}
    \Trp{\H}&:=\left\{T\in\Tr{\H}:T\geq 0\right\},\\
    \State{\H}&:=\left\{T\in\Trp{\H}:\|T\|_{\Trnorm}=\tr{T}=1\right\}.
\end{align*}
An important property of $\Tr{\H}$ is that it is a two sided ideal in
$\LL{\H}$, i.e., for any $\rho\in\Tr{\H}$ and $T\in\LL{\H}$,
we have $\rho T\in\Tr{\H}$ and $T\rho\in\LL{\H}$ (see
\cite[Theorem VI.19]{MathPhys1}).

If for $T\in\LL{\H}$, $\tr{T^*T}<\infty$, then we say $T$ is a
\emph{Hilbert Schmidt operator}
and denote the corresponding space by $\HSS{\H}$. This is
a Hilbert space when equipped with the inner product
$\inp{A}{B}_{\HSnorm}:=\tr{A^*B}$ and the induced norm
$\|T\|_{\HSnorm}:=\sqrt{\inp{T}{T}_{\HSnorm}}$. Note that the product of two Hilbert
Schmidt operators is always trace class.
The introduced spaces compare as follows
\begin{align*}
    &\Tr{\H}\subset\HSS{\H}\subset\LL{\H},\\
    &\|T\|_{\Lnorm}\leq\|T\|_{\HSnorm}\leq\|T\|_{\Trnorm},
\end{align*}
where the inclusions are strict for infinite dimensional Hilbert spaces.

If $T\in\LL{\H}$ is a compact operator, it can be decomposed as
\begin{align*}
    T=\sum_{k=1}^\infty \sigma_k\inp{\cdot}{\varphi_k}_{\H}\psi_k,
\end{align*}
where $\{\varphi_k\}_{k\in\N}$, $\{\psi_k\}_{k\in\N}$ are orthonormal systems
and $\{\sigma_k\}_{k\in\N}$ is a non-increasing sequence of positive numbers.
This is called the \emph{Schmidt decomposition} or
the \emph{singular value decomposition (SVD)}
and is an important
tool for low-rank approximation. The numbers $\sigma_k$ are
called \emph{singular values}.
All trace class and Hilbert Schmidt operators
are compact (but not vice versa).

In quantum mechanics states are modeled by so called \emph{density matrices}
$\rho\in\State{\H}$. We refer to them as \emph{density operators},
to emphasize the
fact that these are, in general, not matrices in this work.
A state is called \emph{pure} if it can be written
as a one dimensional projection, otherwise it is called \emph{mixed}.
I.e., for a pure state there exists $\psi\in\H$ with
$\|\psi\|_{\H}=1$ and $\rho=\inp{\cdot}{\psi}_{\H}\psi$. In general
states are convex combinations of one dimensional
projections of the form
\begin{align*}
    \rho=\sum_{k=1}^\infty \lambda_k\rho_k,
\end{align*}
with positive numbers $\lambda_k$, summing to one. This is a simple
consequence of the spectral decomposition.
The projections $\rho_k$ can be taken to be orthogonal to each other such that
$\tr{\rho}=\sum_{k=1}^\infty\lambda_k=1$. The numbers $\lambda_k$ have a
natural interpretation as probabilities and $\rho$ as a statistical
mixture of pure quantum states.

Suppose we split $\H$ as $\H=\H_A\otimes\H_B$. Then,
the \emph{partial trace} $\ptr{A}{\cdot}:\Tr{\H}\rightarrow\Tr{\H_B}$ is defined
as the unique trace class operator such that for any $E\in\LL{\H_B}$ and any
$T\in\Tr{\H}$
\begin{align*}
    \tr{\ptr{A}{T}E}=\tr{T(\mb I_A\otimes E)}.
\end{align*}
This is useful in order to describe states of subsystems. I.e., if
$\rho\in\State{\H}$ is a density operator, then $\rho_B=\ptr{A}{\rho}
\in\State{\H_B}$ is
also a density operator, describing the state of the subsystem $B$.

We use the shorthand notation
\begin{align*}
    \H_{i,j}:=\mc H_i\otimes\cdots\otimes\mc H_j,\quad 1\leq i\leq j\leq d.
\end{align*}
Suppose $\rho\in\State{\H}$ is a pure state described by $\psi\in\H$. Then,
applying the Hilbert Schmidt decomposition w.r.t.\ the bipartite cut
$\H=\H_{1,j}\otimes\H_{j+1,d}$, we can write
\begin{align*}
    \psi=\sum_{k=1}^{r_j}\sigma_k^j v_k^j\otimes w_k^j,
\end{align*}
where $v_k^j\in\H_{1,j}$, $w_k^j\in\H_{j+1,d}$ and $r_j\in\N\cup\{\infty\}$.
We will frequently use the notation $\sigma_k^j$ for singular values of such
a bipartite cut. The ranks $r_j$ are the TT ranks of $\psi$ since the
chosen cuts $\H=\H_{1,j}\otimes\H_{j+1,d}$ for $1\leq j\leq d-1$ correspond to
the structure of the TT format. Approximability within MPS/TT hinges
on the decay of these particular singular values $\sigma_k^j$ or,
equivalently, on the scaling of the TT ranks $r_j$ for a fixed approximation
accuracy.
\begin{definition}[Approximability]
    We say a function or a state is \emph{approximable} when,
    for a fixed accuracy, the TT ranks grow at most polynomially in $d$.
\end{definition}

For the state of the subsystem on $\H_{1,j}$ we have
the identity
\begin{align*}
    \rho_{1,j}=\ptr{\H_{j+1,d}}{\rho}=\sum_{k=1}^\infty
    (\sigma_k^j)^2\inp{\cdot}{v_k^j}_{\H_{1,j}}v_k^j.
\end{align*}
Thus, we have a correspondence between the probabilities $\lambda_k$
of the state $\rho_{1,j}\in\State{\H_{1,j}}$ and the singular values $\sigma_k^j$ of
$\psi\in\H$.

If $\H=\H_A\otimes\H_B$,
we say an operator $T:\dom{T}\rightarrow\H$ is \emph{supported} on $\H_A$ or
simply $A$, if there exists an operator $T_A:\dom{T_A}\rightarrow\H_A$, such
that
\begin{align*}
    T=T_A\otimes\mb I_B,
\end{align*}
and we write $\supp{T}=\H_A$.

We evolve operators in time according to the standard Heisenberg picture:
for a self-adjoint operator (system Hamiltonian) $H:\dom{H}\rightarrow\H$
and any other operator $T$, we define the time evolution of $T$ as
\begin{align}\label{eq:evolve}
    T(t):=\exp{\i Ht}T\exp{-\i Ht},\quad t\in\R.
\end{align}

Finally, for a state $\rho$, the \emph{von Neumann entropy} is defined as
\begin{align*}
    S(\rho):=-\tr{\rho\log(\rho)},
\end{align*}
and the \emph{R\'{e}nyi entropy} is defined as
\begin{align*}
    S^\alpha(\rho):=\frac{1}{1-\alpha}\log_2(\tr{\rho^\alpha}),
\end{align*}
for $\alpha>0$, $\alpha\neq 1$. One can recover the von Neumann entropy
from the R\'{e}nyi entropy in the limit $\alpha\searrow 1$.

\section{Entropy and Area Laws}\label{sec:entropy}
For an overview of entanglement area laws we refer to \cite{PlenioAL}.
 In the following we introduce the
notion of a general $D$-dimensional area law. For the rest of this work
we focus on 1D area laws.

Let $X$ be a graph representing a constellation of particles.
For the purpose of this introduction and this chapter it is sufficient
to assume $X\subset \Z^D$, i.e., a quantum lattice system, and
that $X$ is finite, i.e., $\#X<\infty$.
Each point
in $X$ represents a separable complex
Hilbert space that acts as the phase space of the
particle corresponding to that point.
We denote the Hilbert space of the entire system by
\begin{align*}
    \X_X:=\bigotimes_{\beta\in X}\X_\beta,
\end{align*}
where $\beta$ are the vertices of $X$.
Throughout this work all tensor product
spaces are equipped and completed w.r.t.\ the canonical norm
(see \eqref{eq:caninp}).

The interactions in this system are given by the system Hamiltonian that we
can write in the general form
\begin{align*}
    H_X=\sum_{I\subset X}\Phi(I),
\end{align*}
where $\Phi(I)$ models interactions within the subset $I$. In this general
form we allow for the
interactions to be trivial, i.e., either $\Phi(I)=0$; or
only one-site operations are present (no interaction): for $a,\;b\in X$
\begin{align*}
    \Phi(a\cup b)=H_a\otimes H_b,
\end{align*}
where $H_a$ and $H_b$ act only on sites $a$ and $b$ respectively.

Given a subset $I\subset X$, we define the set $\partial I$ to be all points in
$I$ that, according to the system Hamiltonian $H_X$, have a non-trivial
interaction with points in the complement $X\setminus I$.
If the current state of the system is described
by $\rho\in\State{\X_X}$, then the state of a subsystem
$I\subset X$ is
described by the partial trace $\rho_I=
\ptr{X\setminus I}{\rho}$. The von Neumann
entropy of any subsystem is then
\begin{align*}
    \s{\rho_I}=-\tr{\rho_I\log_2(\rho_I)}.
\end{align*}
In principle, we could use any entropy measure to formulate area laws
(see \cite{PlenioIntro}).

\begin{definition}[Area Law]
    A pure state $\rho\in\State{\X_X}$
    is said to satisfy
    a \emph{$D$-dimensional\footnote{We emphasize that there is no
    direct relation between $d$ - the number of particles/dimensions - and
    $D$ - the dimension of the area law.} area law}\index{area law} if
    for any $I\subset X$
    the entanglement entropy scales
    proportional to the boundary of $I$, i.e.,
    $\s{\rho_I}\sim\card{\partial I}$.
\end{definition}

This is to be contrasted with
volume laws that state entropy scales as the volume
of $I$,  $\s{\rho_I}\sim\card{I}$. This formulation already suggests that such
laws can be formulated for continuous regions $I$ and $X$, e.g., in quantum
field theory. Indeed, such laws
were originally motivated by similar observations in black hole physics
(see \emph{Bekenstein-Hawking} area law).
For our purposes non-relativistic quantum mechanics with finitely many
particles will suffice.

As an illustration, see
Figure \ref{fig:alaws}. These are examples of 1D and 2D
area laws. In 1D we are
considering chains and thus an area law is particularly simple,
$\s{\rho_I}\sim 2$. On a 2D lattice,
if $I$ encloses $N$ particles,
an area law states $\s{\rho_I}\sim \sqrt{N}$.

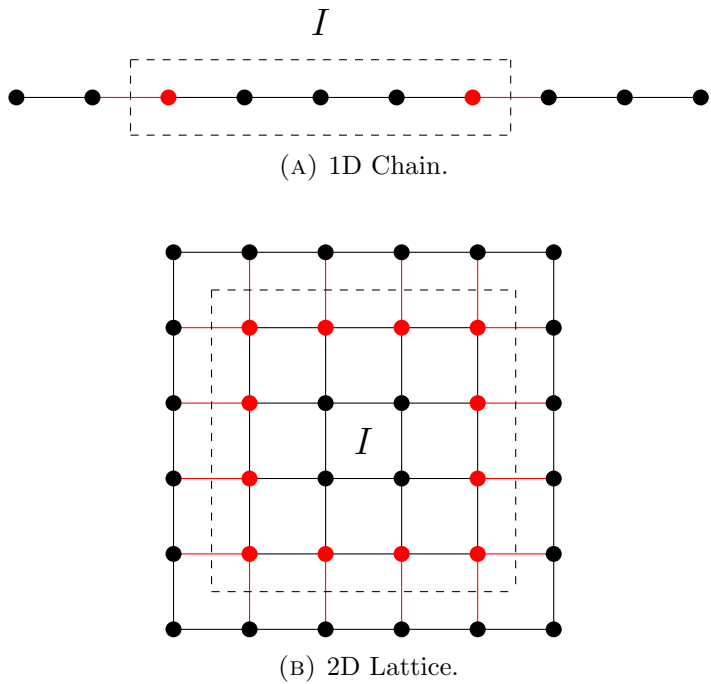
\begin{figure}[h]
    \centering
    \begin{subfigure}{.7\textwidth}
      \centering
      \begin{tikzpicture}
    \coordinate (1) at (0,0);
    \coordinate (2) at (1,0);
    \coordinate (3) at (2,0);
    \coordinate (4) at (3,0);
    \coordinate (5) at (4,0);
    \coordinate (6) at (5,0);
    \coordinate (7) at (6,0);
    \coordinate (8) at (7,0);
    \coordinate (9) at (8,0);
    \coordinate (10) at (9,0);

    \coordinate (11) at (1.5, 0.5);
    \coordinate (12) at (1.5, -0.5);
    \coordinate (13) at (6.5, 0.5);
    \coordinate (14) at (6.5, -0.5);

    \coordinate (15) at (4, 1);

    \draw (1)--(2);
    \draw[red] (2)--(3);
    \draw (3)--(4);
    \draw (4)--(5);
    \draw (5)--(6);
    \draw (6)--(7);
    \draw[red] (7)--(8);
    \draw (8)--(9);
    \draw (9)--(10);

    \draw[dashed] (11)--(12);
    \draw[dashed] (11)--(13);
    \draw[dashed] (12)--(14);
    \draw[dashed] (13)--(14);

    \fill (1) circle (3pt);
    \fill (2) circle (3pt);
    \fill[red] (3) circle (3pt);
    \fill (4) circle (3pt);
    \fill (5) circle (3pt);
    \fill (6) circle (3pt);
    \fill[red] (7) circle (3pt);
    \fill (8) circle (3pt);
    \fill (9) circle (3pt);
    \fill (10) circle (3pt);

    \node at (15) {\Large $I$};
\end{tikzpicture}
      \caption{1D Chain.}
    \end{subfigure}
    \vspace{2em}

    \begin{subfigure}{.7\textwidth}
      \centering
      \begin{tikzpicture}
        \coordinate (11) at (0,0);
        \coordinate (21) at (0,-1);
        \coordinate (31) at (0,-2);
        \coordinate (41) at (0,-3);
        \coordinate (51) at (0,-4);
        \coordinate (61) at (0,-5);

        \coordinate (12) at (1,0);
        \coordinate (22) at (1,-1);
        \coordinate (32) at (1,-2);
        \coordinate (42) at (1,-3);
        \coordinate (52) at (1,-4);
        \coordinate (62) at (1,-5);

        \coordinate (13) at (2,0);
        \coordinate (23) at (2,-1);
        \coordinate (33) at (2,-2);
        \coordinate (43) at (2,-3);
        \coordinate (53) at (2,-4);
        \coordinate (63) at (2,-5);

        \coordinate (14) at (3,0);
        \coordinate (24) at (3,-1);
        \coordinate (34) at (3,-2);
        \coordinate (44) at (3,-3);
        \coordinate (54) at (3,-4);
        \coordinate (64) at (3,-5);

        \coordinate (15) at (4,0);
        \coordinate (25) at (4,-1);
        \coordinate (35) at (4,-2);
        \coordinate (45) at (4,-3);
        \coordinate (55) at (4,-4);
        \coordinate (65) at (4,-5);

        \coordinate (16) at (5,0);
        \coordinate (26) at (5,-1);
        \coordinate (36) at (5,-2);
        \coordinate (46) at (5,-3);
        \coordinate (56) at (5,-4);
        \coordinate (66) at (5,-5);

        \coordinate (1) at (0.5,-0.5);
        \coordinate (2) at (0.5,-4.5);
        \coordinate (3) at (4.5,-0.5);
        \coordinate (4) at (4.5,-4.5);

        \coordinate (5) at (5.5,0.5);
        \coordinate (6) at (2.5,-2.5);

        \draw (11)--(12);
        \draw (12)--(13);
        \draw (13)--(14);
        \draw (14)--(15);
        \draw (15)--(16);

        \draw[red] (21)--(22);
        \draw (22)--(23);
        \draw (23)--(24);
        \draw (24)--(25);
        \draw[red] (25)--(26);

        \draw[red] (31)--(32);
        \draw (32)--(33);
        \draw (33)--(34);
        \draw (34)--(35);
        \draw[red] (35)--(36);

        \draw[red] (41)--(42);
        \draw (42)--(43);
        \draw (43)--(44);
        \draw (44)--(45);
        \draw[red] (45)--(46);

        \draw[red] (51)--(52);
        \draw (52)--(53);
        \draw (53)--(54);
        \draw (54)--(55);
        \draw[red] (55)--(56);

        \draw (61)--(62);
        \draw (62)--(63);
        \draw (63)--(64);
        \draw (64)--(65);
        \draw (65)--(66);

        \draw (11)--(21);
        \draw (21)--(31);
        \draw (31)--(41);
        \draw (41)--(51);
        \draw (51)--(61);

        \draw[red] (12)--(22);
        \draw (22)--(32);
        \draw (32)--(42);
        \draw (42)--(52);
        \draw[red] (52)--(62);

        \draw[red] (13)--(23);
        \draw (23)--(33);
        \draw (33)--(43);
        \draw (43)--(53);
        \draw[red] (53)--(63);

        \draw[red] (14)--(24);
        \draw (24)--(34);
        \draw (34)--(44);
        \draw (44)--(54);
        \draw[red] (54)--(64);

        \draw[red] (15)--(25);
        \draw (25)--(35);
        \draw (35)--(45);
        \draw (45)--(55);
        \draw[red] (55)--(65);

        \draw (16)--(26);
        \draw (26)--(36);
        \draw (36)--(46);
        \draw (46)--(56);
        \draw (56)--(66);

        \draw[dashed] (1)--(2);
        \draw[dashed] (1)--(3);
        \draw[dashed] (2)--(4);
        \draw[dashed] (3)--(4);

        \fill (11) circle (3pt);
        \fill (21) circle (3pt);
        \fill (31) circle (3pt);
        \fill (41) circle (3pt);
        \fill (51) circle (3pt);
        \fill (61) circle (3pt);

        \fill (12) circle (3pt);
        \fill[red] (22) circle (3pt);
        \fill[red] (32) circle (3pt);
        \fill[red] (42) circle (3pt);
        \fill[red] (52) circle (3pt);
        \fill (62) circle (3pt);

        \fill (13) circle (3pt);
        \fill[red] (23) circle (3pt);
        \fill (33) circle (3pt);
        \fill (43) circle (3pt);
        \fill[red] (53) circle (3pt);
        \fill (63) circle (3pt);

        \fill (14) circle (3pt);
        \fill[red] (24) circle (3pt);
        \fill (34) circle (3pt);
        \fill (44) circle (3pt);
        \fill[red] (54) circle (3pt);
        \fill (64) circle (3pt);

        \fill (15) circle (3pt);
        \fill[red] (25) circle (3pt);
        \fill[red] (35) circle (3pt);
        \fill[red] (45) circle (3pt);
        \fill[red] (55) circle (3pt);
        \fill (65) circle (3pt);

        \fill (16) circle (3pt);
        \fill (26) circle (3pt);
        \fill (36) circle (3pt);
        \fill (46) circle (3pt);
        \fill (56) circle (3pt);
        \fill (66) circle (3pt);

        \node at (6) {\Large $I$};
\end{tikzpicture}
      \caption{2D Lattice.}
    \end{subfigure}
    \caption{1D and 2D area laws.}
    \label{fig:alaws}
\end{figure}

Corresponding to the structure of $X$ and interactions in $H_X$,
one can tailor tensor formats
in order to efficiently represent such states on $\X$. See
\cite{TN, BSU} for an overview. For 1D chains the corresponding tensor
format is referred to as a Matrix Product State (MPS) in the physics
community, or Tensor Train (TT) in mathematics. In 1D systems there is a
strong link between area laws and low-rank approximability. Due to this and the
fact that, unlike in general tensor networks, the best approximation
problem in TT is well posed, we focus
on 1D area laws.


\subsection{Entropy Scaling and Approximation}\label{sec:entapp}

We begin by illustrating the connection between entropy scaling in $d$
and low-rank approximation in 1D systems.
It can be summarized by Table \ref{tab:summ}. Recall the definition of the
R\'{e}nyi entropy
\begin{align*}
    S^\alpha(\rho):=(1-\alpha)^{-1}\log_2(\tr{\rho^\alpha}),\quad
    \alpha>0,\;\alpha\neq 1.
\end{align*}
And, since $\lim_{\alpha\searrow 1}S^\alpha(\rho)=S(\rho)$, we use
$S^{\alpha=1}$ to denote the von Neumann entropy.

\begin{table}
    \centering
    \begin{tabular}{l | c | c | c | c}
        Entropy        &const               &$\log_2(\card{I})$         &$(\card{I})^{p<1}$        & $\card{I}$ \\
        \hline\hline
        $S^{\alpha<1}$ &\cellcolor{green!65}&\cellcolor{green!65}&\cellcolor{blue!25}&\cellcolor{blue!25}\\
        $S^{\alpha=1}$ &\cellcolor{blue!25} &\cellcolor{blue!25} &\cellcolor{blue!25}&\cellcolor{red!65} \\
        $S^{\alpha>1}$ &\cellcolor{blue!25} &\cellcolor{blue!25} &\cellcolor{red!65} &\cellcolor{red!65} \\
    \end{tabular}
    \caption{Entropy vs.\ Approximability 1D systems
    \cite{EntSim}. $\card{I}$ denotes the length of the subsystem.
    \textcolor{green!65}{$\blacksquare$} Approximable,
    \textcolor{blue!25}{$\blacksquare$} undetermined (both are possible),
    \textcolor{red!65}{$\blacksquare$} not approximable.}
    \label{tab:summ}
\end{table}

Throughout this work we will use the shorthand notation
\begin{align}\label{eq:xij}
    \X_{i,j}:=\bigotimes_{k=i}^j\X_k,
\end{align}
where, as mentioned above, all tensor product spaces are equipped
and completed w.r.t.\ the canonical norm.
The following result is based\footnote{The main idea remains the same with
a more rigorous proof.} upon \cite{EntSim} and \cite{VCMPS}.

\begin{proposition}\label{prop:alaw_app}
    Let $\uu\in\X$,
    $\norm{\uu}{\X}=1$,
    $\rho:=\inp{\cdot}{\uu}_{\X}\uu\in\State{\X}$,
    $\rho_{1,j}:=\ptr{\X_{j+1,d}}{\rho}\in\State{\X_{1,j}}$.
    Suppose $\s[\alpha]{\rho_{1,j}}<\infty$
    if $\alpha<1$ (for $\alpha>1$ the entropy is clearly finite since
    $\rho\in\State{\X}$).
    Let $\varepsilon_j(r)$ denote the truncation error for a bipartite cut
    $\X\cong\X_{1,j}\otimes\X_{j+1,d}$,
    $\varepsilon^2_j(r)=\sum_{k=r+1}^\infty(\sigma_k^j)^2$. Then,
    \begin{align}
        \s[\alpha]{\rho_{1,j}}&\geq\frac{\alpha}{1-\alpha}
        \log_2\left(\frac{\varepsilon_j^2(r)}{\alpha}\right)+
        \log_2\left(\frac{r-1}{1-\alpha}\right),
        \quad 0<\alpha<1,\label{eq:entL}\\
        \s[\alpha]{\rho_{1,j}}&\leq
        \frac{\alpha}{1-\alpha}\log_2(1-\varepsilon^2_j(r))+\log_2(r),
        \quad \alpha>1\label{eq:entU}.
    \end{align}
\end{proposition}

Before we proceed with the proof, we require the following lemma.

\begin{lemma}\label{lemma:schur}
    Let $\alpha>0$, $\alpha\neq 1$
    and $\seq{a}=\set{a_k:\,k\in\N}$, $\seq{b}=\set{b_k:\,k\in\N}$
    be two non-negative,
    non-increasing sequences such that
    \begin{align*}
        \sum_{k=1}^\infty a_k^\alpha<\infty,\quad
        \sum_{k=1}^\infty b_k^\alpha<\infty,\quad
        \sum_{k=1}^\infty a_k=\sum_{k=1}^\infty b_k<\infty,
    \end{align*}
    and for any $m\in\N$,
    \begin{align*}
        \sum_{k=1}^ma_k\geq\sum_{k=1}^mb_k.
    \end{align*}
    Then,
    \begin{align*}
        \s[\alpha]{\seq{a}}\leq\s[\alpha]{\seq{b}},
    \end{align*}
    where the entropy of a sequence is given as
    $\s[\alpha]{\seq{a}}=\frac{1}{1-\alpha}\log_2(\sum_{k=1}^\infty a_k^\alpha)$.
\end{lemma}

\begin{proof}
    The above property holds for finite sequences since entropy is Schur\index{Schur concave}
    concave. We thus assume that there is no such $n\in\N$ such that $a_n=0$
    or $b_n=0$ and reduce the above lemma to the case of finite sequences.
    We define truncated sequences $\tilde{\seq{a}}^n$ and $\tilde{\seq{b}}^n$ for
    any $n\in\N$
    such that
    \begin{align*}
        \tilde{a}_i&=a_i,\quad\text{for }1\leq i\leq n+1,\quad
        \tilde{a}_{n+i}=a_{n+1},\quad\text{for }1\leq i\leq m,\\
        \tilde{a}_{n+m+1}&=s-\sum_{k=1}^na_k-ma_{n+1},\quad
        \tilde{a}_{n+m+1+k}=0,\quad\text{for }k\in\N,
    \end{align*}
    where $m$ is chosen such that $0\leq\tilde{a}_{n+m+1}\leq a_{n+1}$.
    For $l=l(n)\in\N$
    \begin{align*}
        \tilde{b}_i&=b_i,\quad\text{for }1\leq i\leq l+1,\quad
        \tilde{b}_{l+i}=b_{l+1},\quad\text{for }1\leq i\leq p,\\
        \tilde{b}_{l+p+1}&=s-\sum_{k=1}^lb_k-pb_{l+1},\quad
        \tilde{b}_{l+p+1+k}=0,\quad\text{for }k\in\N,
    \end{align*}
    where $l(n)$ is large enough such that $b_{l+1}\leq a_{n+1}$ and $p$ is
    chosen analogously as above such that
    $0\leq\tilde{b}_{l+p+1}\leq b_{l+1}$.

    Both $\tilde{\seq{a}}^n$ and $\tilde{\seq{b}}^n$ are non-increasing
    sequences with finitely many non-zero terms and, by construction,
    $\tilde{\seq{a}}^n$ majorizes $\tilde{\seq{b}}^n$ such that we can conclude
    $\s[\alpha]{\tilde{\seq{a}}^n}\leq\s[\alpha]{\tilde{\seq{b}}^n}$.
    It remains to
    show that
    \begin{align*}
        \lim_{n\rightarrow\infty}\s[\alpha]{\tilde{\seq{a}}^n}
        =\s[\alpha]{\seq{a}}\quad\text{and}\quad
        \lim_{n\rightarrow\infty}\s[\alpha]{\tilde{\seq{b}}^n}
        =\s[\alpha]{\seq{b}}.
    \end{align*}

    Since the $\log_2$-function is continuous, it suffices to show that the
    argument converges. We thus estimate
    \begin{align*}
        \left|\sum_{k=1}^\infty a_k^\alpha-\sum_{k=1}^\infty\tilde{a}_k^\alpha
        \right|\leq\sum_{k=n+2}^{n+m}|a_k^\alpha-a^\alpha_{n+1}|+
        |a^\alpha_{n+m+1}-\tilde{a}^\alpha_{n+m+1}|+\sum_{k=n+m+2}^\infty
        a_k^\alpha.
    \end{align*}
    The second and third term obviously converge for $n\rightarrow\infty$.
    For the first term we consider $\alpha<1$ (otherwise the statement
    is straightforward)
    \begin{align*}
        \sum_{k=n+2}^{n+m}|a_k^\alpha-a^\alpha_{n+1}|\leq
        \sum_{k=n+2}^\infty a_k^\alpha+ma_{n+1}^\alpha.
    \end{align*}
    The first term converges while for the second term we obtain
    \begin{align*}
        ma_{n+1}^\alpha&=ma_{n+1}a_{n+1}^{\alpha-1}\leq
        \left(\sum_{k=n+1}^\infty a_k\right)a_{n+1}^{\alpha-1}
        \leq\sum_{k=n+1}^\infty a_ka_{n+1}^{\alpha-1}
        \overset{\alpha<1}{\leq}\sum_{k=n+1}^\infty a_ka_k^{\alpha-1}\\
        &\leq\sum_{k=n+1}^\infty a_k^{\alpha}\overset{n}
        {\longrightarrow} 0.
    \end{align*}
    Similarly for $\tilde{\seq{b}}^n$ and thus the statement follows.
\end{proof}

\begin{proof}[Proof of Proposition \ref{prop:alaw_app}]
    We follow the arguments from \cite[Lemma 2]{VCMPS} and
    \cite{EntSim} with some adjustments.
    The idea is that we want to bound $\s[\alpha]{\rho_{1,j}}$ by
    exploiting the fact that R\'{e}nyi entropies are Schur concave. I.e.,
    construct a sequence that majorizes or is majorized by
    the eigenvalues of $\rho_{1,j}$ and
    compute the entropy for the former explicitly.

    \textbf{Case $0<\alpha<1$}.
    Let $p:=\varepsilon_j^2(r)$. Then, $\sum_{k=1}^r(\sigma_k^j)^2=1-p$.
    For $h>0$, consider the candidate sequence
    \begin{align*}
        \lambda_1&=(1-p)-(r-1)h,\quad\lambda_i=h\quad\text{for }2\leq i\leq n-1,\\
        \lambda_n&=p-h(n-r-3)=:\theta,\quad
        \lambda_k=0\quad\text{for } k>n,
    \end{align*}
    where $n\geq r+1$ is chosen such that $0\leq\theta\leq h$.
    For $h=0$, we simply set $\lambda_1=1$ and $\lambda_k=0$ for $k>1$.
    Our goal is to show
    that this sequence majorizes $\set{(\sigma_k^j)^2:\,k\in\N}$.

    If $h=0$, then $\set{\lambda_k:\,k\in\N}$ is
    non-increasing,
    \begin{align*}
        1=\sum_{k=1}^\infty\lambda_k=
        \sum_{k=1}^\infty(\sigma^j_k)^2,
    \end{align*}
    and
    \begin{align*}
        1=\sum_{k=1}^m\lambda_k\geq\sum_{k=1}^m(\sigma^j_k)^2,
    \end{align*}
    for any
    $m\in\N$. I.e., $\set{(\sigma_k^j)^2:\,k\in\N}$ is trivially
    majorized by $\set{\lambda_k:\,k\in\N}$.

    Thus, consider $h>0$. In order to ensure the sequence
    is non-increasing, we require $\lambda_1\geq h$ and thus
    $h\leq\frac{1-p}{r}$.

    Next,
    \begin{align*}
        \lambda_1=(1-p)-(r-1)h\geq (1-p)-\frac{r-1}{r}(1-p)=(1-p)\frac{1}{r}>0.
    \end{align*}
    For $1\leq k\leq n-1$
    \begin{align*}
        \lambda_1-\lambda_k=\lambda_1-h=(1-p)-rh\geq (1-p)-\frac{r}{r}(1-p)=0.
    \end{align*}
    And finally by the choice of $n$, $0\leq\theta\leq h$. Thus,
    $\set{\lambda_k:\,k\in\N}$ is non-increasing and by construction sums to 1.

    For $1\leq m\leq r$,
    \begin{align*}
        &\sum_{k=1}^m\lambda_k+(r-m)h=1-p=\sum_{k=1}^r(\sigma_k^j)^2,\\
        &\Leftrightarrow\sum_{k=1}^m\lambda_k-(\sigma_k^j)^2=
        \sum_{k=m+1}^r(\sigma_k^j)^2-(r-m)h\geq(r-m)((\sigma_r^j)^2-h)=0,
    \end{align*}
    with equality for $m=r$.
    For $r<m<n$, we have
    \begin{align*}
        \sum_{k=1}^m\lambda_k-(\sigma_k^j)^2\geq (1-p)+(m-r)h-
        (1-p)-(m-r)(\sigma_k^j)^2=0,
    \end{align*}
    and for $m\geq n$,
    \begin{align*}
        \sum_{k=1}^m\lambda_k-(\sigma_k^j)^2&=\sum_{k=1}^r\lambda_k
        -(\sigma_k^j)^2
        +\sum_{k=r+1}^m\lambda_k-\sum_{k=r+1}^m(\sigma_k^j)^2\\
        &\geq\sum_{k=r+1}^n\lambda_k-\sum_{k=r+1}^\infty(\sigma_k^j)^2=0.
    \end{align*}
    Thus, $\set{\lambda_k:\,k\in\N}$ majorizes $\set{(\sigma^j_k)^2:\,k\in\N}$
    and
    by Lemma \ref{lemma:schur}, we get
    \[\s[\alpha]{\set{(\sigma_k^j)^2:\,k\in\N}}
    \geq\s[\alpha]{\set{\lambda_k:\,k\in\N}}.\]
    The entropy of the majorizing sequence can be computed as
    \begin{align*}
        \sum_{k=1}^\infty\lambda_k^\alpha=[(1-p)-(r-1)h]^\alpha+
        [(r-1)+\lfloor\frac{p}{h}\rfloor]h^\alpha+
        [p-h(n-r-3)]^\alpha.
    \end{align*}

    Estimating from below
    \begin{align*}
        \sum_{k=1}^\infty\lambda_k^\alpha&\geq h^\alpha+(r-1)h^\alpha+
        (p/h-1)h^\alpha+[p-h(n-r-3)]^\alpha\\
        &\geq (r-1)h^\alpha+ph^{\alpha-1}.
    \end{align*}
    We further estimate from below by minimizing the expression on the right.
    I.e., for
    \begin{align*}
        f(h)&=(r-1)h^\alpha+ph^{\alpha-1},\quad
        f'(h)=\alpha(r-1)h^{\alpha-1}+(\alpha-1)ph^{\alpha-2}=0,\\
        &\Leftrightarrow\quad h=\frac{(1-\alpha)p}{\alpha(r-1)}=:h^*.
    \end{align*}
    Moreover,
    \[f''(h^*)=(1-\alpha)^{\alpha-1}p^\alpha(r-1)^{1-\alpha}\alpha^{-\alpha}>0,\]
    for $0<\alpha<1$. Plugging in $\s[\alpha]{\rho_{1,j}}\geq
    \frac{1}{1-\alpha}\log_2(f(h^*))$, gives \eqref{eq:entL}.

    \textbf{Case $\alpha>1$}. The maximizing distribution for this case is
    straightforward. Pick
    \begin{align*}
        \lambda_k&=\frac{1-p}{r},\quad 1\leq k\leq r,\quad
        \lambda_k=(\sigma_k^j)^2,\quad k>r.
    \end{align*}
    Clearly, $\sum_{k=1}^m\lambda_k\leq\sum_{k=1}^m(\sigma_k^j)^2$ for any
    $m\in\N$ with equality for $m\geq r$. Thus,
    \begin{align*}
        \s[\alpha]{\rho_{1,j}}&\leq\frac{1}{1-\alpha}\log_2
        \left(r\frac{(1-p)^\alpha}{r^\alpha}+\sum_{k=r+1}^\infty
        (\sigma_k^j)^2\right)\leq
        \frac{1}{1-\alpha}\log_2
        \left(\frac{(1-p)^\alpha}{r^{\alpha-1}}\right)\\
        &=
        \frac{\alpha}{1-\alpha}\log_2(1-p)+\log_2(r),
    \end{align*}
    which shows \eqref{eq:entU} and completes the proof.
\end{proof}

Proposition \ref{prop:alaw_app} immediately gives.
\begin{corollary}
    For $\uu\in\X$,
    if the R\'{e}nyi entropy $\s[\alpha]{\rho_{1,j}}$,
    $0<\alpha<1$, scales at most as some power of $\log_2(j)$, then
    $u$ is
    approximable within the TT format, i.e., TT ranks grow at most
    polynomially with the dimension $d$.
    On the other hand, if the R\'{e}nyi entropy $\s[\alpha]{{\rho}_{1,j}}$,
    $\alpha>1$ scales as some power of $j$, then $u$ is
    not approximable within the TT format, i.e., TT ranks grow exponentially
    with $d$.
\end{corollary}

Next, we want to show that a lower bound on the von Neumann entropy implies
a lower bound on rank growth. To this end, we require the ability to
approximate the entropy of a given state by entropies of states with finite
ranks. I.e., we require continuity of the von Neumann entropy. Unfortunately,
the von Neumann entropy is continuous only on a small subset of
$\Tr{\X}$
that is nowhere dense. In fact, the set of states with infinite von Neumann
entropy is trace-norm dense in $\Tr{\X}$.

However, under certain ``physical''
assumptions, one can show continuity of $\s{\cdot}$.
This is will be discussed in
greater detail in Section \ref{sec:entconv}. For now we assume continuity
as given. Then, we can show

\begin{proposition}
    \index{entropy!von Neumann}
    For $r\in\N$, let $\varepsilon_j^2(r)$ denote the best approximation error
    for an $r$-term truncation in the bipartite cut
    $\X\cong\X_{1,j}
    \otimes\X_{j+1,d}$, and
    $\rho_{1,j}^r$ the corresponding best approximation.
    If $\s{\rho_{1,j}}\geq c\min\set{j,\,d}$ and
    $g(r):=|\s{\rho_{1,j}^r}-\s{\rho_{1,j}}|$, then
    $r\geq\exp{c\min\set{j,\,d/2}-g(r)}$.
\end{proposition}

\begin{proof}
    First, note that we can bound the entropy of a state of rank $r$ by
    $\log_2(r)$. The von Neumann Entropy is Schur concave as any entropy measure
    should be\footnote{A majorizing sequence of probabilities represents
    less uncertainty and
    thus should have smaller entropy in any meaningful measure.},
    since it quantifies chaos or lack of information. Thus,
    the maximum entropy is attained for a uniform distribution
    \begin{align*}
        \s{\rho_{1,j}^r}\leq-
        \sum_{k=1}^r\frac{1}{r}\log_2\left(\frac{1}{r}\right)
        =\log_2(r).
    \end{align*}
    See \cite[Example 2.28]{QI} for more details.

    Thus, following arguments from \cite{EntSim}
    \begin{align*}
        \log_2(r)\geq\s{\rho_{1,j}^r}\geq\s{\rho_{1,j}}-
        |\s{\rho_{1,j}}-\s{\rho_{1,j}^r}|,
    \end{align*}
    which completes the proof.
\end{proof}
Of course, the above estimate is useful only
if $\lim_{r\rightarrow\infty}g(r)=0$ which also
assumes that $\s[\alpha]{\rho_{1,j}}<\infty$.
For the undetermined region of Table \ref{tab:summ}, we refer to the examples
provided in \cite{EntSim}. These carry over to the infinite dimensional case.
We repeat one such example as a demonstration.
E.g., a state that can be efficiently approximated with TT can have linearly
growing R\'{e}nyi entropy. To see this, assume w.l.o.g.\footnote{This
assumption merely simplifies notation.}
$\X_j=\X_i=:\X^{1d}$ for all $1\leq i,\,j\leq d$ and
let $\varphi_1$, $\varphi_2$ and $\varphi_3$ be
orthonormal in $\X^{1d}$. Then, for $0\leq p_d\leq 1$, set
\begin{align*}
    \psi_{2d}:=\sqrt{1-p_d}(
    \underbrace{\varphi_3\otimes\ldots\otimes\varphi_3}_{2d\text{ times}})
    +\sqrt{\frac{p_d}{2^d}}\sum_{i_1,\ldots,i_d\in\set{1,\,2}}
    (\varphi_{i_1}\otimes\ldots\otimes\varphi_{i_d})\otimes(\varphi_{i_1}\otimes\ldots\otimes\varphi_{i_d}).
\end{align*}
Clearly $\norm{\psi_{2d}}{\X}=1$. We can set $p_d:=1/d$ which, by
construction, implies $\psi_{2d}$ converges to a rank-one state with
growing $d$.
I.e., by a simple rank-one approximation we obtain
\begin{align*}
    \norm{\psi_{2d}-\sqrt{1-p_d}(
    \underbrace{\varphi_3\otimes\ldots\otimes\varphi_3}_{2d\text{ times}})}{\X}
    =\sqrt{p_d}\overset{d\rightarrow\infty}{\longrightarrow}0.
\end{align*}
On the other hand, for $1\leq j\leq d$, the density operator of the subsystem
can be computed as
\begin{align*}
    \rho_{1,j}&=(1-p_d)\rho^{j}+\frac{p_d}{2^j}\sum_{i_1,\ldots,i_j\in\set{1,\,2}}
    \rho^{i_1,\ldots,i_j},\\
    \rho^j&:=
    \inp{\cdot}{\underbrace{\varphi_3\otimes\ldots\otimes\varphi_3}_{j\text{ times}}}_{\X_{1,j}}
    \underbrace{\varphi_3\otimes\ldots\otimes\varphi_3}_{j\text{ times}},\\
    \rho^{i_1,\ldots,i_j}&:=
    \inp{\cdot}{\varphi_{i_1}\otimes\ldots\otimes\varphi_{i_j}}_{\X_{1,j}}
    \varphi_{i_1}\otimes\ldots\otimes\varphi_{i_j}
\end{align*}
and therefore the R\'{e}nyi entropy for $0<\alpha<1$ is
\begin{align*}
    \s[\alpha]{\rho_{1,j}}=\frac{1}{1-\alpha}\log_2([1-p_d]^\alpha+
    2^{(1-\alpha)j}p_d^\alpha)\geq j-\frac{\alpha}{1-\alpha}\log_2(d)
    =j+\frac{2\alpha}{1-\alpha}\log_2(\varepsilon),
\end{align*}
for $\varepsilon:=\sqrt{p_d}=\sqrt{1/d}$.
For examples of the other parts of the undetermined
region of Table \ref{tab:summ}
we refer to \cite{EntSim}.

We conclude this section by a brief discussion of entropy discontinuity.
For $\uu\in\X$ and
$\rho:=\inp{\cdot}{\uu}_{\X}\uu\in\State{\X}$,
$\s[\alpha]{\rho_{1,j}}$ for
some $1\leq j\leq d$ is not necessarily finite for $0<\alpha\leq 1$.
In such cases entropy is no longer a useful measure of
approximability. We are thus not
certain if or to what extent questions like ``Does infinite entropy imply
inapproximability?'' or ``Does approximability imply finite entropy?'' make
sense. For $\alpha>1$ the R\'{e}nyi entropy is always finite. In this case
we can certify that a faster than logarithmic scaling implies inapproximability.


\subsection{Entropy Convergence}\label{sec:entconv}
In this section, we consider the question of finite entropy or entropy
continuity. This is important not only for considerations in Section
\ref{sec:entapp}, but for Section \ref{sec:alaw} as well. Since one of
our
results is an upper bound for the von Neumann entropy in the infinite
dimensional setting, we first have to consider for what states entropy
makes sense in the first place.

The set of states with infinite entropy is dense in
$\Tr{\X}$
(see \cite[section II.D]{Wehrl}). From a purely analytic standpoint, entropy
is finite if the singular values of bipartite cuts converge fast enough.
Put more precisely, given an algebraic decay of the singular values, we
obtain

\begin{proposition}\label{prop:convrate}
    For any $1\leq j\leq d$, if $\sigma_k^j\lesssim k^{-s}$ for
    $s>\frac{1}{2\alpha}$, $0<\alpha\leq 1$,
    then $\s[\alpha]{\rho_{1,j}}<\infty$.
\end{proposition}

\begin{proof}
    For the von Neumann entropy, by \cite[Equation (45)]{Baccetti}, it is
    sufficient to show that there exists some $\delta>0$ such that
    $\sum_{k=1}^\infty(\sigma_k^j)^2k^\delta<\infty$. I.e., by assumption,
    \begin{align*}
        \sum_{k=1}^\infty(\sigma_k^j)^2k^\delta\lesssim
        \sum_{k=1}^\infty k^{-2s+\delta}<\infty,
    \end{align*}
    for some
    $0<\delta<2s-1$.
    For $0<\alpha<1$, we obtain
    \begin{align*}
        \s[\alpha]{\rho_{1,j}}\lesssim\sum_{k=1}^\infty k^{-2s\alpha}<\infty,
    \end{align*}
    since $-2s\alpha<-1$.
\end{proof}

By virtue of the fact that $\rho\in\Tr{\X}$,
$\s[\alpha]{\rho_{1,j}}<\infty$
for any $\alpha>1$. The requirement in Proposition \ref{prop:convrate} is
more useful as a necessary condition: if entropy fails to be finite,
it tells us how ``slow'' the decay rate must be. However, since we are
ultimately interested in the relation between entropy scaling and
approximation, it is not useful as a criteria to decide a priori the
approximability of a system.

To this end, we discuss a set of conditions frequently assumed in the physics
literature (see \cite{Wehrl, PlenioInfDim}). There are essentially two
difficulties that appear in the setting of infinite dimensional Hilbert spaces,
not encountered in finite dimensional systems.

Firstly, it is possible for the expected energy described by the system
Hamiltonian to be infinite in a given state. This is due to the fact that in
the infinite dimensional setting Hamiltonians are generally unbounded, just
like any differential operator. It can be shown that if the entropy of a given
state is infinite, then so must be the expected energy of that state
(see \cite[section II.D]{Wehrl}).

Secondly, in the finite dimensional case it is straight forward to determine
the state of maximal entropy (maximal chaos). It is simply the state with
a uniform distribution of probabilities, i.e., the
density matrix is the identity operator times a normalization constant
$1/d$.

This does not work anymore in infinite dimensions, since
such a state is no longer normalizable. Instead, the state of maximal entropy
is given by the \emph{Gibbs} state, a well-known equilibrium distribution from
statistical mechanics. This state is not unique and does not need to exist,
but rather depends on the inverse temperature (parameter $\beta$ from
Proposition \ref{prop:finentropy}) which leads to different
expected energies.

Assuming the existence of the Gibbs state is equivalent to assuming the
system Hamiltonian has purely discrete spectrum and the eigenvalues
diverge ``fast enough'' (see \cite[Thms XIII.16, XIII.67]{MathPhys4}).
Physically, it means that we assume the existence of
a thermodynamic limit at any temperature. However, even in a
physically meaningful setting,
the Gibbs state does not have to exist: neither on physical grounds nor on
mathematical. We will return to this issue in Section \ref{sec:summ}.
For now we formulate a model setting in which the von Neumann entropy is
finite.

\begin{proposition}\label{prop:finentropy}
    \index{entropy!von Neumann}
    Suppose we are given a self-adjoint PDE operator
    $H:\dom{H}\rightarrow\X$ such that
    \begin{align*}
        H=H_{1,j}\otimes\id_{j+1,d}+R_j,\quad R_j\geq 0,
    \end{align*}
    for any $1\leq j\leq d$, where
    $H_{1,j}:\dom{H_{1,j}}\rightarrow\X_{1,j}$
    is self-adjoint and $\id_{j+1,d}:\X_{j+1,d}\rightarrow
    \X_{j+1,d}$ is the identity operator.
    For $\rho\in\State{\X}$, assume that
    $\tr{\rho H}=:E<\infty$ and
    $\exp{-\beta H}\in\Tr{\X}$ for all $\beta>0$.
    Then, $\s{\rho_{1,j}}<\infty$ for any $1\leq j\leq d$.
\end{proposition}

\begin{proof}
    The idea of the proof is as follows. Since the expected energy of the whole
    system is finite, then so is the expected energy of the subsystems. For
    the given energy $E$, we can choose an inverse temperature $\beta>0$,
    such that the corresponding Gibbs state has the same expected energy.
    Since the Gibbs state has maximal entropy (for fixed $E$), this provides
    the upper bound.

    \textbf{I}. For any $1\leq j\leq d$, it holds that
    \begin{align*}
        \infty>E&=\tr{\rho H}=\tr{\rho H_{1,j}\otimes\id_{j+1,d}+\rho R_j}
        \overset{R_j\geq 0}{\geq}\tr{\rho H_{1,j}\otimes\id_{j+1,d}}\\
        &=\tr{\rho_{1,j}H_{1,j}}.
    \end{align*}

    \textbf{II}. Our assumptions imply that
    $\exp{-\beta H}$ is compact and thus
    $H$ must have purely discrete spectrum which is bounded from below and has
    diverging eigenvalues such that
    \begin{align*}
        \tr{\exp{-\beta H}}=\sum_{k=1}^\infty\exp{-\beta\lambda_k}<\infty.
    \end{align*}
    From hereon we assume the eigenvalues of $H$,
    $\set{\lambda_k:\,k\in\N}$, are ordered in non-decreasing order.

    The existence of the Gibbs state for any inverse temperature
    $\beta>0$ implies the expected that energy is finite for any $\beta>0$.
    To see
    this, note that since $\lim_{k\rightarrow\infty}\lambda_k=\infty$,
    there exists an $N\in\N$
    such that $\exp{\frac{\beta}{2}\lambda_k}\geq\lambda_k$ for $k\geq N$. Thus,
    \begin{align*}
        \sum_{k=N}^\infty\exp{-\beta\lambda_k}\lambda_k
        \leq\sum_{k=N}^\infty\exp{-\beta\lambda_k}
        \exp{\frac{\beta}{2}\lambda_k}=\sum_{k=N}^\infty\exp{-\frac{\beta}{2}
        \lambda_k}<\infty.
    \end{align*}
    By applying elementary functional calculus we compute
    $\tr{\exp{-\beta H}H}=\sum_{k=1}^\infty\exp{-\beta\lambda_k}\lambda_k
    <\infty.$

    \textbf{III}. Thus, we have ensured the existence of the Gibbs state
    \begin{align}\label{eq:gibbsdef}
        \rho_\beta:=\exp{-\beta H}/Z,\quad Z:=\tr{\exp{-\beta H}},
    \end{align}
    with finite energy
    \begin{align*}
        E(\beta)=\tr{\rho_\beta H}=\left(\sum_{k=1}^\infty
        \exp{-\beta\lambda_k}\right)^{-1}\sum_{k=1}^\infty\exp{-\beta\lambda_k}
        \lambda_k.
    \end{align*}
    The function $E(\beta)$ is continuous for $\beta\in(0,\infty)$. To see this,
    take any $\beta>0$ and consider the interval $[\beta/2,2\beta]$. Since
    the sequence
    $\set{\lambda_k:\;k\in\N}$ diverges, there exists $N_1\in\N$, such that
    $\exp{-\frac{\beta}{2}\lambda_k}\lambda_k$ is monotonically decreasing for
    $k\geq N_1$ and there exists $N_2=N_2(\varepsilon)\in\N$, such that
    \begin{align*}
        \sum_{k=N_2}^\infty\exp{-\frac{\beta}{2}\lambda_k}\lambda_k\leq
        \varepsilon.
    \end{align*}
    Thus, for any $\varepsilon>0$, we can take
    $N=N(\varepsilon)=\max\set{N_1,\,N_2(\varepsilon)}$ such that for
    $M\geq N$ and any $\tilde{\beta}\in[\beta/2,2\beta]$
    \begin{align*}
        \sum_{k=M}^\infty\exp{-\tilde{\beta}\lambda_k}\lambda_k
        \leq
        \sum_{k=M}^\infty\exp{-\frac{\beta}{2}\lambda_k}\lambda_k
        \leq\varepsilon.
    \end{align*}
    Thus, this series converges uniformly in $\beta$ on compact sets in
    $\R^+$ and is therefore continuous for any $\beta>0$. An analogous
    argument shows the same for the series
    $\sum_{k=1}^\infty\exp{-\beta\lambda_k}$. Hence, $E(\beta)$ is continuous.

    \textbf{IV}. The idea behind this part of the proof is as
    follows: for $\beta\rightarrow 0$, i.e., temperature
    $T\rightarrow\infty$, higher energies become more probable such that we
    anticipate for the expected energy of a Gibbs state
    $E(\beta)\rightarrow\infty$. On the other hand, for $\beta\rightarrow\infty$
    ($T\rightarrow 0$), we expect probabilities to cluster around the ground
    state energy $\lambda_1$.

    Since the sequence $\set{\lambda_k:\;k\in\N}$ diverges,
    for any $C>0$, there exists $N=N(C)\in\N$ such that
    $\frac{1}{N}\sum_{k=1}^N\lambda_k\geq C$ and, clearly,
    $N\rightarrow\infty$ as $C\rightarrow\infty$.

    As was shown in \textbf{III}, both the nominator and denominator
    in $E(\beta)$ are series that converge uniformly in $\beta$.
    Thus, for any $\varepsilon>0$,
    there exists $\tilde{\beta}>0$, such that for any $\beta\leq\tilde{\beta}$
    \begin{align*}
        E(\beta)\geq\frac{(\sum_{k=1}^N\lambda_k)-\varepsilon}{(\sum_{k=1}^N1)+
        \varepsilon}=\frac{1}{N+\varepsilon}\sum_{k=1}^N\lambda_k-
        \frac{\varepsilon}{N+\varepsilon}\geq
        \frac{N}{N+\varepsilon}C-\frac{\varepsilon}{N+\varepsilon}.
    \end{align*}
    Since this is possible for any $\varepsilon$ and any $C>0$, we have
    $\lim_{\beta\rightarrow 0}E(\beta)=\infty$.

    On the other hand, let $p_k:=\exp{-\beta\lambda_k}/Z(\beta)$. Then,
    \begin{align*}
        \left|E(\beta)-\lambda_1\right|=
        \sum_{k=2}^\infty p_k(\beta)(\lambda_k-\lambda_1).
    \end{align*}
    Since this series converges, for any $\varepsilon>0$, there exists
    $N\in\N$, such that $\sum_{k=N}^\infty p_k(\beta)(\lambda_k-\lambda_1)
    \leq\varepsilon$. We can choose $N$ independently of $\beta$, since
    for $\beta\rightarrow\infty$ the series converges faster and
    thus $N$ gets smaller.
    Due to the convergence of $\set{p_2(\beta),\ldots,p_{N-1}(\beta)}$
    for $\beta\rightarrow\infty$,
    there exists $\tilde{\beta}>0$ such that for $\beta\geq\tilde{\beta}$,
    $p_k(\beta)(\lambda_k-\lambda_1)\leq\frac{\varepsilon}{N-2}$ for
    $2\leq k\leq N-2$. Thus,
    \begin{align*}
        \sum_{k=2}^\infty p_k(\beta)(\lambda_k-\lambda_1)=
        \sum_{k=2}^{N-1}p_k(\beta)(\lambda_k-\lambda_1)+
        \sum_{k=N}^\infty p_k(\beta)(\lambda_k-\lambda_1)\leq 2\varepsilon.
    \end{align*}
    Since this is possible for any $\varepsilon>0$, we conclude
    $\lim_{\beta\rightarrow\infty}E(\beta)=\lambda_1$.

    \textbf{V}. Finally, take any $1\leq j\leq d$. By \textbf{I},
    $E_j:=\tr{\rho_{1,j}H_{1,j}}\leq E=\tr{\rho H}<\infty$. By
    \textbf{III-IV}, there exists $\beta>0$ such that
    $\tr{\rho_\beta^jH_{1,j}}=E_j$.

    In \cite[section I.B.5, inequality (1.41)]{Wehrl}
    it is shown that if $\rho$ is a state with $E=\tr{\rho H}$ and
    $\rho_\beta$ is defined as in \eqref{eq:gibbsdef} with
    $E=\tr{\rho_\beta H}=\tr{\rho H}$, then
    $S(\rho)\leq S(\rho_\beta)$. Applying this to our problem we get
    $\s{\rho_{1,j}}\leq \s{\rho_{\beta}^j}=\beta E_j+\log_2(Z)<\infty$.
    This completes the proof.
\end{proof}

\section{Ground State Approximability}\label{sec:alaw}
We turn to the main result of this work.
We illustrate how local interactions
in a Hamiltonian operator imply that
the ground state
can be approximated by operators which have a small overlap
in the support.
Under an additional assumption on the
approximand (see Assumption \ref{ass:tm}), we show that
the entropy of the ground state does not scale with the dimension.
Consequently, by the considerations in Section \ref{sec:entapp},
such eigenfunctions enjoy favorable separability properties within the TT format.
To stay consistent with the typical notation in physics,
we will slightly adapt indices in this Section by shifting all summations to
start from $0$ such that the first eigenvalue will be denoted by $\lambda_0$.

\subsection{Ground State Density Operator}
Before we can proceed with the entropy bound,
we require some
preparations. The key ingredient will be Theorem \ref{thm:obolor}, which
essentially states that the ground state density operator can be approximated
by a product of 3 local operators with overlapping support, where the error
converges exponentially in the length of the overlap.
Indeed, Theorem \ref{thm:obolor} is interesting in its own right and we
consider it to be the main contribution of this work.
This approximation is possible if
the Hamiltonian operator satisfies the following properties.

\begin{assumption}\label{ass:op}
    Let $H:\dom{H}\rightarrow\X$ be a densely defined
    self-adjoint (possibly unbounded) operator.
    \begin{enumerate}
        \item\label{ass:local} (Locality).
        We assume $H$ can be decomposed as
        \begin{align*}
            H=\sum_{j=1}^{d-1}H_{j,j+1},
        \end{align*}
        where each $H_{j,j+1}$ is supported on
        $\X_{j,j+1}$
        (see \eqref{eq:xij}).

    \item\label{ass:gap} (Gap).
        We assume the spectrum is bounded from below and
        the ground state is unique with a non-vanishing spectral gap
        \begin{align}\label{eq:spgap}
            \Delta E:=\lambda_1-\lambda_0>0.
        \end{align}
        We denote the ground state wave function
        by $\uu_0$ and the
        corresponding density operator by
        \begin{align}\label{eq:kuidiot}
            \rho^0:=\inp{\cdot}{\uu_0}_{\X}\uu_0.
        \end{align}

    \item\label{ass:int} (Finite Interaction Strength).
        We assume for all $1\leq j\leq d-1$, $H_{j,j+1}=H_j+H_{j+1}+
        \Phi_{j,j+1}$, where $H_j$ and $H_{j+1}$ are possibly unbounded
        operators supported on $\X_j$ and $\X_{j+1}$, respectively,
        and $\Phi_{j,j+1}$ is a uniformly bounded
        operator\footnote{That models interactions between particles
        $j$ and $j+1$.}
        supported on $\X_{j,j+1}$. I.e., there exists a constant $J$
        such that
        \begin{align}\label{eq:intstr}
            \norm{\Phi_{j,j+1}}{\Lnorm}\leq J,
        \end{align}
        for all $1\leq j\leq d-1$.

    \item\label{ass:comm} (Bounded Commutators).
        The commutators of the neighboring
        interaction and single site operators
        are uniformly bounded\footnote{Or can be uniquely extended to bounded
        operators.},
        i.e.,
        \begin{align*}
            \norm{\comm{\Phi_{j,j+1}}{H_{j+1}}}{\Lnorm}\leq J\quad\text{and}\quad
            \norm{\comm{H_j}{\Phi_{j, j+1}}}{\Lnorm}\leq J,
        \end{align*}
        for all $1\leq j\leq d-1$.

    \item\label{ass:self} (Self-Adjoint).
        The interaction and single site operators $\Phi_{j,j+1}$ and
        $H_{j}$ are self-adjoint.
    \end{enumerate}
\end{assumption}

\begin{remark}\label{rem:ass}
    Assumption \eqref{ass:local} means we only consider local 2-site
    interactions. Our results would remain unchanged for $N$-site interactions,
    for a fixed $N$. The point is that the complexity of approximating
    an eigenfunction scales exponentially with $N$ and not $d$. Moreover,
    we expect similar
    results could be obtained for long range interactions that decay
    sufficiently fast.

    We require Assumption \eqref{ass:self} since the proof heavily relies
    on the spectral decomposition. One could possibly
    generalize the proofs presented
    here to sectorial operators. We are not certain to what extent
    approximability actually depends on the form
    of the resolvent/spectrum of the
    operator in $\C$.

    Assumption \eqref{ass:gap} is necessary for an area law to hold.
    Systems with degenerate ground states are at a
    quantum critical point and have been observed to exhibit
    divergent entanglement entropies (see
    \cite{PlenioAL, VidalQCP, Calabrese}).
    However, uniqueness of the ground state is not necessary for the
    main estimate in Theorem \ref{thm:obolor}.

    Assumptions \eqref{ass:int} and \eqref{ass:comm} are required for the
    application of Lieb-Robinson bounds, i.e., finite speed information
    propagation. There are essentially two difficulties when considering
    information propagation for dynamics prescribed by an unbounded
    operator.

    First, unlike with classic
    Lieb-Robinson bounds (see \cite{LBOriginal}),
    bounded local operators do not have to remain local when evolved via the
    unitary operator $\exp{\i Ht}$ (see \cite{SuperSonice}).
    This can be remedied as in \cite{LBInfDim,
    NearComm} by,
    e.g., assuming the interactions in $H$ are of a certain type,
    such as bounded.
    Hence, we require Assumption \eqref{ass:int}.

    Second, when applying time dynamics to an unbounded local operator, it is
    not clear in which sense the operator remains \emph{approximately}
    local. Thus, Assumption \eqref{ass:comm} ensures that the non-local part
    is bounded.

    However, we essentially require only an application of
    Lieb-Robinson. Although Assumptions \eqref{ass:int} and \eqref{ass:comm}
    are certainly sufficient, they are perhaps not necessary.
\end{remark}

\begin{example}[Nearest Neighbor Interaction (NNI)]\label{ex:nni}
    We provide an example of how the general structure of such an NNI
    Hamiltonian might look like.
    Perhaps the most famous example of an NNI Hamiltonian
    is the \emph{Ising model} (see \cite{Ising}).

    In this work we
    consider infinite dimensional Hilbert spaces and
    unbounded Hamiltonians.
    A typical example to keep in mind is
    $\X=\bigotimes_{j=1}^d\X_j=\bigotimes_{j=1}^d\Lp{\R^n}{\C}$,
    where
    $n\in\set{1,\,2,\,3}$ if $H$ is to model a physical phenomenon.

    Let the Hamiltonian operator be given
    as
    \begin{align*}
        H=-\Delta+V.
    \end{align*}
    The Laplacian $\Delta$ is the one-site unbounded operator
    where $H_j=-\frac{\partial^2}{\partial x_j^2}$. The potential
    $V$ contains the bounded interaction operators. E.g.,
    $V=\sum_{j=1}^{d-1}\Phi_{j,j+1}$, where $\Phi_{j,j+1}:\X\rightarrow\X$ is
    a bounded operator such as
    \begin{align*}
        (\Phi_{j,j+1}\uu)(x)&=c(x_j,\,x_{j+1})\uu(x),\quad\text{or}\\
        (\Phi_{j,j+1}\uu)(x)&=\int_{\R^{2n}}\kappa(x_j,\,x_{j+1},\,y_j,\,y_{j+1})
        \uu(x_1,\,\ldots,\,y_j,\,y_{j+1},\,\ldots,\,x_d)\d (y_j,\,y_{j+1}),
    \end{align*}
    where $c(\cdot)$ is a bounded coefficient function and $\kappa(\cdot)$ is an
    integral kernel.

    We would have to check that a given Hamiltonian has a gap above the ground
    state and if the ground state is unique. Note that the gap property is
    much more important than uniqueness, since the latter is only
    necessary for the area law in Theorem \ref{thm:alaw} and not for the
    approximation in Theorem \ref{thm:obolor}.
    Spectral properties and uniqueness of ground states
    have been extensively studied before and we refer to, e.g.,
    \cite[Chapter XIII]{MathPhys4} for more details.
\end{example}

We begin with a lemma that shows how we can approximately express the
ground state projector through the Hamiltonian operator. This will provide the
necessary link between the local operator structure and the local structure
of the density operator.

\begin{lemma}\label{lemma:pq}
    Let Assumption \ref{ass:op} \eqref{ass:gap} hold. Assume w.l.o.g.\
    that $\lambda_0=0$ and let $\rho^0$
    denote the corresponding density operator (see \eqref{eq:kuidiot}).
    Then, for any $q>0$ and
    \begin{align}\label{eq:pq}
        \rho^q:=\frac{1}{\sqrt{2\pi q}}\int_{-\infty}^\infty
        \exp{\i Ht}\exp{-\frac{t^2}{2q}}\d t,
    \end{align}
    we have
    \begin{align*}
        \norm{\rho^q-\rho^0}{\Lnorm}
        \leq\exp{-\frac{1}{2}(\Delta E)^2q},
    \end{align*}
    with $\Delta E$ from \eqref{eq:spgap}.
\end{lemma}

\begin{proof}
    The operator $U(t):=\exp{\i Ht}\exp{-\frac{t^2}{2q}}$ is strongly continuous
    for all $t\in\R$. Thus, a finite integral of $U(t)$ is well defined.
    For any $\uu\in\X$
    \begin{align*}
        \lim_{c\rightarrow\infty}\norm{\frac{1}{\sqrt{2\pi q}}
        \int_{-c}^cU(t)\uu\d t}{\X}\leq
        \lim_{c\rightarrow\infty}\norm{\uu}{\X}\frac{1}{\sqrt{2\pi q}}
        \int_{-c}^c\exp{-\frac{t^2}{2q}}\d t=\norm{\uu}{\X}.
    \end{align*}
    Thus, the integral \eqref{eq:pq} is well defined.

    Since $H$ is self-adjoint, we have the spectral decomposition
    \begin{align*}
        H=\int_{\sigma(H)}\lambda\d P(\lambda),
    \end{align*}
    where $P:\sigma(H)\rightarrow\LL{\X}$ is a projection valued measure.
    Due to the gap assumption, we get that $\rho^0=P(\lambda_0)$.

    Applying functional calculus for self-adjoint operators
    \begin{align*}
        \exp{\i Ht}=\int_{\sigma(H)}\exp{\i\lambda t}\d P(\lambda).
    \end{align*}
    Equation
    \eqref{eq:pq} is to be interpreted as the unique operator such that
    for any $\uu\in\X$
    \begin{align*}
        \inp{\uu}{\frac{1}{\sqrt{2\pi q}}\int_{-\infty}^\infty U(t)\uu\d t}_{\X}
        &=
        \frac{1}{\sqrt{2\pi q}}\int_{-\infty}^\infty
        \exp{-\frac{t^2}{2q}}\inp{\uu}{U(t)\uu}_{\X}\d t\\
        &=
        \frac{1}{\sqrt{2\pi q}}\int_{-\infty}^\infty
        \exp{-\frac{t^2}{2q}}\int_{\sigma(H)}\exp{\i\lambda t}\d P_{\uu}(\lambda)\d t,
    \end{align*}
    where $P_{\uu}(\cdot)=\inp{\uu}{P(\cdot)\uu}_{\X}$ and the equality
    follows from the linearity and continuity of the $\X$-inner product.
    For the last integral we can apply Fubini's Theorem
    for general product measures.
    This allows us to write
    \begin{align*}
        \rho^q&=\frac{1}{\sqrt{2\pi q}}\int_{-\infty}^\infty \int_{\sigma(H)}
        \exp{\i\lambda t}\exp{-\frac{t^2}{2q}}\d P(\lambda)\d t\\
        &\overset{\text{gap}}{=}\frac{1}{\sqrt{2\pi q}}\int_{-\infty}^\infty
        \rho^0\exp{-\frac{t^2}{2q}}+\int_{\sigma(H)\setminus\set{\lambda_0}}
        \exp{i\lambda t}\exp{-\frac{t^2}{2q}}\d P(\lambda)\d t\\
        &\overset{\text{Fubini}}{=}
        \rho^0+\frac{1}{\sqrt{2\pi q}}\int_{\sigma(H)\setminus\set{\lambda_0}}
        \int_{-\infty}^\infty\exp{\i\lambda t}\exp{-\frac{t^2}{2q}}\d t\d P(\lambda).
    \end{align*}
    The last term is the Fourier transform of the density of the
    normal distribution. Thus,
    \begin{align*}
        \norm{\rho^q-\rho^0}{\Lnorm}=
        \norm{\int_{\sigma(H)\setminus\set{\lambda_0}}
        \exp{-\frac{1}{2}\lambda^2 q}\d P(\lambda)}{\Lnorm}
        \leq\exp{-\frac{1}{2}(\Delta E)^2q},
    \end{align*}
    which completes the proof.
\end{proof}

Next, we want to approximate $H$ by a sum of three local operators, where each
operator
approximately annihilates the ground state. To this end, we apply Hasting's
quasi-adiabatic continuation
technique (see \cite{QA, Hastings}),
which was also studied
in \cite{NearComm} in the infinite dimensional setting\footnote{There the
authors considered this technique in order to describe states that belong to
the same phase.}.

\begin{lemma}\label{lemma:ann}
    Suppose Assumption \ref{ass:op} (\eqref{ass:local}-\eqref{ass:comm})
    holds.

    For a fixed $l\in\N$ and a
    fixed $1+l\leq j\leq d-2-l$,
    \begin{align*}
        H_L&:=\sum_{k\leq j-l-2}H_{k,k+1},\quad
        H_B:=\sum_{j-l-1\leq k\leq j+l+1}H_{k,k+1},\quad
        H_R:=\sum_{k\geq j+l+2}H_{k,k+1}.
    \end{align*}
    W.l.o.g.\ let $\inp{\uu_0}{H_L\uu_0}_{\X}=
    \inp{\uu_0}{H_B\uu_0}_{\X}=\inp{\uu_0}{H_R\uu_0}_{\X}=0$
    and $\lambda_0=0$. Then, for any $q>0$ and
    \begin{align*}
        \tilde{H}_L&:=\frac{1}{\sqrt{2\pi q}}\int_{-\infty}^\infty H_L(t)\exp{-\frac{t^2}{2q}}\d t,\\
        \tilde{H}_B&:=\frac{1}{\sqrt{2\pi q}}\int_{-\infty}^\infty H_B(t)\exp{-\frac{t^2}{2q}}\d t,\\
        \tilde{H}_R&:=\frac{1}{\sqrt{2\pi q}}\int_{-\infty}^\infty H_R(t)\exp{-\frac{t^2}{2q}}\d t,
    \end{align*}
    where $H_{\cdots}(t)$ is given as in \eqref{eq:evolve} by, e.g.,
    \begin{align*}
        H_L(t):=\exp{\i Ht}H_L\exp{-\i Ht},
    \end{align*}
    we have
    \begin{align*}
        \norm{\tilde{H}_L\uu_0}{\X}&\leq
        3J^2(\Delta E)^{-1}\exp{-\frac{1}{2}(\Delta E)^2q},\\
        \norm{\tilde{H}_B\uu_0}{\X}&\leq
        3J^2(\Delta E)^{-1}\exp{-\frac{1}{2}(\Delta E)^2q},\\
        \norm{\tilde{H}_R\uu_0}{\X}&\leq
        3J^2(\Delta E)^{-1}\exp{-\frac{1}{2}(\Delta E)^2q}.
    \end{align*}
    The constant $J$ is the interaction strength from \eqref{eq:intstr}.
\end{lemma}

\begin{proof}
    By the same arguments as in Lemma \ref{lemma:pq}, the integrals are
    well defined. Next, application to the ground state yields
    \begin{align*}
        H_L(t)\uu_0&=\exp{\i Ht}H_L\int_{\sigma(H)}\exp{-\i\lambda t}\d P(\lambda)
        \uu_0\overset{\lambda_0=0}{=}\exp{\i Ht}H_L\uu_0
        =
        \int_{\sigma(H)}\exp{\i\lambda t}\d P(\lambda)H_L\uu_0\\
        &=\inp{\uu_0}{H_L\uu_0}\uu_0+\int_{\sigma(H)\setminus\set{\lambda_0}}
        \exp{\i\lambda t}\d P(\lambda)H_L\uu_0.
    \end{align*}
    Thus,
    \begin{align*}
        \tilde{H}_L\uu_0&=\frac{1}{\sqrt{2\pi q}}\int_{-\infty}^\infty
        H_L(t)\uu_0\exp{-\frac{t^2}{2q}}\d t\\
        &=\int_{\sigma(H)\setminus\set{\lambda_0}}\frac{1}{\sqrt{2\pi q}}
        \int_{-\infty}^\infty\exp{\i\lambda t}\exp{-\frac{t^2}{2q}}\d t\d P(\lambda)
        H_L\uu_0\\
        &=
        \int_{\sigma(H)\setminus\set{\lambda_0}}\exp{-\frac{1}{2}\lambda^2 q}
        \d P(\lambda)H_L\uu_0.
    \end{align*}
    On the other hand,
    \begin{align*}
        H\tilde{H}_L\uu_0=\int_{\sigma(H)\setminus\set{\lambda_0}}
        \lambda\exp{-\frac{1}{2}\lambda^2q}\d P(\lambda)H_L\uu_0.
    \end{align*}
    Hence,
    \begin{align}\label{eq:hhl}
        \norm{H\tilde{H}_L\uu_0}{\X}\geq \Delta E\norm{
        \int_{\sigma(H)\setminus\set{\lambda_0}}\exp{-\frac{1}{2}\lambda^2q}
        \d P(\lambda)H_L\uu_0}{\X}=\Delta E\norm{\tilde{H}_L\uu_0}{\X}.
    \end{align}
    Next, since $H\uu_0=0$,
    \begin{alignat*}{2}
        H\tilde{H}_L\uu_0&=&&(H\tilde{H}_L-\tilde{H}_LH)\uu_0
        =\int_{\sigma(H)\setminus\set{\lambda_0}}\exp{-\frac{1}{2}\lambda^2q}
        \d P(\lambda)\comm{H}{H_L}\uu_0\\
        &=&&\int_{\sigma(H)\setminus\set{\lambda_0}}\exp{-\frac{1}{2}\lambda^2q}
        \d P(\lambda)\comm{H_{j-l-1,j-l}}{H_{j-l-2,j-l-1}}\uu_0\\
        &=&&\int_{\sigma(H)\setminus\set{\lambda_0}}\exp{-\frac{1}{2}\lambda^2q}
        \d P(\lambda)
        \bigg\{
        \comm{H_{j-l-1}}{(\Phi_{j-l-2,j-l-1}-\Phi_{j-l-1,j-l})}\\
        & &&+\comm{\Phi_{j-l-1,j-l}}{\Phi_{j-l-2,j-l-1}}\bigg\}\uu_0,
    \end{alignat*}
    where the last equality follows from Assumption \eqref{ass:int}.
    And thus
    \begin{align*}
        \norm{H\tilde{H}_L\uu_0}{\X}\leq 3J^2\exp{-\frac{1}{2}(\Delta E)^2q}.
    \end{align*}
    Together with \eqref{eq:hhl}
    \begin{align*}
        \norm{\tilde{H}_L\uu_0}{\X}\leq
        3J^2(\Delta E)^{-1}\exp{-\frac{1}{2}(\Delta E)^2q},
    \end{align*}
    and analogously for $\tilde{H}_B$, $\tilde{H}_R$. This completes the
    proof.
\end{proof}

\begin{remark}
    Note that the above bound depends explicitly only on $q$, $\Delta E$
    and $J$. In fact, more
    precisely, the bound depends on an estimate for
    \begin{align*}
        \norm{\comm{H}{H_L}}{\Lnorm},
    \end{align*}
    and the latter was assumed in \eqref{eq:intstr} to be uniformly bounded,
    i.e., in particular it is independent of $j$ or $l$.

    However, the subsequent lemmas will employ Lieb-Robinson bounds that
    depend explicitly on the parameter $l$. Thus, we will eventually use the
    above lemma and choose
    the constant $q$ depending on the spectral gap
    $\Delta E$ and the parameter $l$.
\end{remark}

Next, we show that $\tilde{H}_L$, $\tilde{H}_B$ and $\tilde{H}_R$
are approximately local\footnote{In the sense specified by the following lemma.}.
This is mainly due to
Lieb-Robinson type estimates.
\begin{lemma}\label{lemma:ml}
    Under Assumption \ref{ass:op} (\eqref{ass:local}, \eqref{ass:int},
    \eqref{ass:comm}), there exist local bounded operators
    $\diff_L$, $\diff_B$ and
    $\diff_R$ supported on
    $\X_{j-2l-2, j}$,
    $\X_{j-2l-2,j+2l+3}$ and
    $\X_{j+1, j+2l+3}$,
    respectively, such that for
    \begin{align*}
        M_L&:=H_L+\diff_L,\quad
        M_B:=H_B+\diff_B,\quad
        M_R:=H_R+\diff_R,
    \end{align*}
    there exist constants $c_1>0$, $C_1>0$ such that
    \begin{align*}
        \norm{\tilde{H}_L-M_L}{\Lnorm}&\leq C_1J^2\max\set{q^{1/2},\,q^{3/2}}\exp{-c_1l},\\
        \norm{\tilde{H}_B-M_B}{\Lnorm}&\leq C_1J^2\max\set{q^{1/2},\,q^{3/2}}\exp{-c_1l},\\
        \norm{\tilde{H}_R-M_R}{\Lnorm}&\leq C_1J^2\max\set{q^{1/2},\,q^{3/2}}\exp{-c_1l},
    \end{align*}
    where $q$, $l$ are the parameters from Lemma \ref{lemma:ann} and
    $J$ is the interaction strength from \eqref{eq:intstr}.
\end{lemma}
\begin{proof}
    First, note that we can differentiate $H_L(t)$ to obtain
    \begin{align*}
        \frac{\d}{\d t}H_L(t)&=\frac{\d}{\d t}\exp{\i Ht}H_L\exp{-\i Ht}\\
        &=\exp{\i Ht}\i HH_L\exp{-\i Ht}-\exp{\i Ht}iH_LH\exp{-\i Ht}\\
        &=\exp{\i Ht}\i\comm{H}{H_L}\exp{-\i Ht}
        =:\i\comm{H}{H_L}(t),
    \end{align*}
    and $H_L(0)=H_L$.
    Thus, we can write
    \begin{align}\label{eq:thanks1}
        H_L(t)=H_L+\int_{0}^t\i\comm{H}{H_L}(\tau)\d\tau.
    \end{align}
    By Assumption, $\comm{H}{H_L}$ is bounded and supported on
    $\X_{j-l-2,j-l}$. Consequently, we can write
    \begin{align*}
        \tilde{H}_L=H_L+\frac{1}{\sqrt{2\pi q}}\int_{-\infty}^\infty
        \int_0^t\i\comm{H}{H_L}(\tau)\d\tau\exp{-\frac{t^2}{2q}}\d t.
    \end{align*}
    Since the commutator is bounded and local, and the interactions in
    $H$ are bounded, by \cite[Corollary 2.2]{LBInfDim},
    we know a Lieb-Robinson bound
    applies to $\comm{H}{H_L}$. I.e., there exists a constant (velocity)
    $v\geq 0$ and constants $C>0$, $a>0$, such that
    \begin{align}\label{eq:lrbnd}
        &\norm{\comm{\comm{H}{H_L}(\tau)}{B}}{\Lnorm}
        \leq C\norm{\comm{H}{H_L}}{\Lnorm}
        \norm{B}{\Lnorm}\exp{-a\left\{\dist(\comm{H}{H_L}, B)-v|\tau|
        \right\}},
    \end{align}
    for all bounded and local $B$.
    Thus, by \cite[Lemma 3.2]{NearComm},
    there exists a map $\Pi:\LL{\X}\rightarrow\LL{\X}$ such that
    $\Pi(A)$ is supported on $\X_{j-2l-2,j}$ and, for any $A\in\LL{\X}$
    satisfying \eqref{eq:lrbnd} with $\dist(A, B)\geq l$, we have
    \begin{align*}
        \norm{A-\Pi(A)}{\Lnorm}\leq 2C\norm{A}{\Lnorm}
        \exp{-a\left\{l-v|\tau|
        \right\}}.
    \end{align*}
    Then, using \cite[Theorem 3.4]{NearComm}, we integrate over time
    and further estimate
    \begin{align*}
        &\norm{\int_{0}^t\comm{H}{H_L}(\tau)\d\tau-
        \int_0^t\Pi\left(\comm{H}{H_L}(\tau)\right)
        \d\tau}{\Lnorm}
        \leq
        |t|CJ^2\exp{-a\{l-v|t|\}}.
    \end{align*}

    We define the operator
    \begin{align*}
        \diff_L:=\frac{1}{\sqrt{2}\pi q}\int_{-\infty}^\infty
        \int_0^t \Pi\left(\i\comm{H}{H_L}(\tau)\right)\d\tau\exp{-\frac{t^2}{2q}}
        \d t.
    \end{align*}
    By all of the above,
    this operator is bounded and supported on $\X_{j-2l-2, j}$. Analogously,
    we define $\diff_B$ and $\diff_R$ with supports in
    $\X_{j-2l-2,j+2l+3}$ and $\X_{j+1, j+2l+3}$, respectively.

    What remains is to truncate the tails of the integral to obtain an
    overall error of the same order as the Lieb-Robinson bound. Let
    $T=\frac{l}{2v}$. Then,
    \begin{align}\label{eq:trunc}
        &\norm{\tilde{H}_L^1-M_L}{\Lnorm}=\notag\\
        &=
        \norm{\frac{1}{\sqrt{2\pi q}}\int_{-\infty}^\infty\int_{0}^t
        \left(\i\comm{H}{H_L}(\tau)-
        \Pi\left(\i\comm{H}{H_L}(\tau)\right)\right)\d\tau
        \exp{-\frac{t^2}{2q}}\d t}{\Lnorm}\notag\\
        &=\norm{\frac{1}{\sqrt{2\pi q}}\left(
        \int_{|t|\leq T}\ldots+\int_{|t|>T}\ldots\right)}{\Lnorm}\notag\\
        &\leq
        CJ^2\exp{-al}
        \left(\exp{avT}\frac{1}{\sqrt{2\pi q}}\int_{|t|\leq T}
        |t|\exp{-\frac{t^2}{2q}}\d t+\frac{1}{\sqrt{2\pi q}}\int_{|t|>T}
        |t|\exp{av|t|-\frac{t^2}{2q}}\d t\right).
    \end{align}
    For the first term
    \begin{align*}
        \frac{1}{\sqrt{2\pi q}}\int_{|t|\leq T}
        |t|\exp{-\frac{t^2}{2q}}\d t\leq \sqrt{\frac{q}{2\pi}}.
    \end{align*}
    For the second
    \begin{alignat*}{2}
        \frac{1}{\sqrt{2\pi q}}\int_{|t|>T}
        |t|\exp{av|t|-\frac{t^2}{2q}}\d t
        &=&&\sqrt{\frac{2}{\pi q}}\int_{T}^\infty
        t\exp{avt-\frac{t^2}{2q}}\d t\\
        &=&&q\exp{avT-\frac{T^2}{2q}}\\
        & &&+avq
        \int_T^\infty\exp{avt-\frac{t^2}{2q}}\d t.
    \end{alignat*}
    For the latter term
    \begin{align*}
        \int_T^\infty\exp{avt-\frac{t^2}{2q}}\d t&=
        \int_T^\infty\exp{\frac{(av)^2q}{2}-
        \left(\sqrt{\frac{1}{2q}}t-av\sqrt{\frac{q}{2}}\right)^2}\d t\\
        &=
        \sqrt{2q}\int_{\sqrt{\frac{1}{2q}}T-av\sqrt{\frac{q}{2}}}^\infty
        \exp{-\tau^2}\d\tau\leq
        \sqrt{\frac{q\pi}{2}}\exp{avT-\frac{T^2}{2q}}.
    \end{align*}
    And thus
    \begin{alignat*}{2}
        \norm{\tilde{H}_L-M_L}{\Lnorm}
        &\leq&&
        CJ^2\exp{-al}\times\\
        & &&\bigg(
        \exp{al/2}\sqrt{\frac{q}{2\pi}}
        +q\exp{al/2}\exp{-\frac{T^2}{2q}}
        +avq^{3/2}\sqrt{\frac{\pi}{2}}\exp{al/2}\exp{-\frac{T^2}{2q}}
        \bigg)\\
        &\leq&& CJ^2\exp{-\frac{a}{2}l}
        \max\set{
        \sqrt{\frac{q}{2\pi}},\, q,\, avq^{3/2}\sqrt{\frac{\pi}{2}}}
        \left(1+2\exp{-\frac{l^2}{8v^2q}}\right)\\
        &\leq&& C_1J^2\max\set{q^{1/2},\,q^{3/2}}\exp{-c_1l},
    \end{alignat*}
    with $C_1$ and $c_1$ defined in an obvious way as above.
    This completes the proof.
\end{proof}

We can conclude the existence of the first two operators that
we will need to approximate $\rho^0$.

\begin{lemma}\label{lemma:ol}
    Under Assumption \ref{ass:op}, there exist local, bounded and
    self-adjoint (projection) operators $O_L=O_L(l)$,
    $O_R=O_R(l)$ with the property
    \begin{align*}
        \norm{(O_L-\id)\uu_0}{\X}&\leq \exp{-c_1l/2},\\
        \norm{(O_R-\id)\uu_0}{\X}&\leq \exp{-c_1l/2}.
    \end{align*}
    The operators $O_L$ and $O_R$ have the same support as $M_L$
    and $M_R$, respectively, and $\norm{O_L}{\Lnorm}
    =\norm{O_R}{\Lnorm}=1$.
\end{lemma}

\begin{proof}
    Recall from Lemma \ref{lemma:ml} that we applied Lieb-Robinson to operators
    such as $\i\comm{H}{H_L}(\tau)$. By Assumption \ref{ass:op},
    since $H$ and $H_L$ are self-adjoint, the commutators are bounded and
    one has
    $\adj{\i\comm{H}{H_L}}=-\i\left(H_LH-HH_L\right)=\i\comm{H}{H_L}$
    $\Rightarrow$ we
    can conclude
    that the commutator is self-adjoint. Applying Lieb-Robinson as in
    \cite[Lemma 3.2]{NearComm}, one can construct the local approximation to
    preserve self-adjointness. We briefly elaborate.

    Let $\X=\X_1\otimes \X_2$.
    For the construction of $\Pi:\X\rightarrow\X$
    in \cite[Lemma 3.2]{NearComm}, the authors use an arbitrary state
    $\rho\in\State{\X_2}$, though the resulting bound does not depend on the
    choice of $\rho$. Then, for this state, applying the spectral decomposition,
    we have
    \begin{align*}
        \rho=\sum_{k=1}^\infty\lambda_k\inp{\cdot}{\psi_k}_{\X_2}\psi_k,
    \end{align*}
    for $\set{\psi_k:\;k\in\N}$ orthonormal in $\X_2$. Next, for each $\psi_k$,
    the authors define the map $A_k$ as
    \begin{align*}
        \inp{v}{A_kw}_{\X_1}=\inp{v\otimes\psi_k}{Aw\otimes\psi_k}_{\X},\quad
        \forall v,\,w\in\X_1.
    \end{align*}
    Finally, the map $\Pi(A)$ is defined as
    \begin{align*}
        \Pi(A):=\left(\sum_{k=1}^\infty\lambda_kA_k \right)\otimes\id.
    \end{align*}
    Note that each $A_k$ is self-adjoint if $A$ is self-adjoint. Therefore,
    $\Pi(A)$ is self-adjoint.

    Thus, $M_L$, $M_B$ and $M_R$ can be chosen
    self-adjoint.
    By Lemmas \ref{lemma:ml} and \ref{lemma:ann}, picking
    $q=c_1\frac{2l}{(\Delta E)^2}$, we get
    $\norm{M_L\uu_0}{\X}\leq C_1J^2\max\set{q^{1/2},\,q^{3/2}}\exp{-c_1l}$.
    Moreover, since $M_L$ is
    self-adjoint, there exists a projection valued measure $P(\cdot)$ such that
    \begin{align*}
        \normsq{M_L\uu_0}{\X}=\inp{M_L\uu_0}{M_L\uu_0}_{\X}=
        \inp{\uu_0}{M_L^2\uu_0}_{\X}=\int_{\sigma(M_L)}\lambda^2\d
        P_{\uu_0}(\lambda).
    \end{align*}

    We split the spectrum of $M_L$ as
    \begin{align*}
        \sigma_1(M_L)&:=\set{\lambda\in\sigma(M_L):|\lambda|\leq
        C_1J^2\max\set{q^{1/2},\,q^{3/2}}\exp{-c_1l/2}},\\
        \sigma_2(M_L)&:=\sigma(M_L)\setminus
        \sigma_1(M_L).
    \end{align*}
    Define $O_L$ as
    \begin{align*}
        O_L:=\int_{\sigma_1(M_L)}\d P(\lambda).
    \end{align*}
    Clearly, $O_L$ is a bounded self-adjoint operator with
    $\norm{O_L}{\Lnorm}=1$ and the same support as $M_L$.
    Moreover, by orthogonality of the spectral subspaces
    \begin{align*}
        C_1^2J^4\max\set{q,\,q^{3}}\exp{-c_12l}&\geq \norm{M_L\uu_0}{\X}
        =
        \int_{\sigma_1(M_L)}\lambda^2\d P_{\uu_0}(\lambda)\uu_0+
        \int_{\sigma_2(M_L)}\lambda^2\d P_{\uu_0}(\lambda)\uu_0\\
        &\geq C_1^2 J^4\max\set{q,\,q^{3}}
        \exp{-c_1l}\int_{\sigma_2(M_L)}\d P_{\uu_0}(\lambda).
    \end{align*}
    Thus, $\int_{\sigma_2(M_L)}\d P_{\uu_0}(\lambda)\leq\exp{-c_1l}$.
    Finally, this gives
    \begin{align*}
        \norm{(O_L-\id)\uu_0}{\X}&=
        \norm{\int_{\sigma_2(M_L)}\d P(\lambda)\uu_0}{\X}=
        \left(\int_{\sigma_2(M_L)}\d P_{\uu_0}(\lambda)\right)^{1/2}\\
        &\leq\exp{-c_1l/2}.
    \end{align*}
    Analogously for $O_R$. This completes the proof.
\end{proof}

What remains is a step by step approximation of $\rho^q$ as a product of three
local operators.

\begin{lemma}\label{lemma:tpq}
    Under Assumption \ref{ass:op}, we can further approximate $\rho^0$ as
    \begin{align*}
        \tilde{\rho}^q&:=\frac{1}{\sqrt{2\pi q}}\int_{-\infty}^\infty
        \adj{\oe{A}{0}{t}}\exp{-\frac{t^2}{2q}}O_LO_R\d t,\\
        A(t)&:=\exp{\i (M_L+M_R)t}\i M_B\exp{-\i (M_L+M_R)t},
    \end{align*}
    where $\adj{\oe{A}{0}{t}}$ is the negative time-ordered
    exponential
    and $q=c_1\frac{2l}{(\Delta E)^2}$. We have
    \begin{align*}
        \norm{\tilde{\rho}^q-\rho^0}{\Lnorm}\leq
        C_2J^2\max\set{q,\,q^2}\left(
        \exp{-c_1l/2}+\exp{-c_1l}\right),
    \end{align*}
    for some constant $C_2>0$.
\end{lemma}

\begin{proof}
    Note that $H=\tilde{H}_L+\tilde{H}_B+
    \tilde{H}_R$. By utilizing the estimates from Lemma \ref{lemma:ml}, we have
    \begin{align}\label{eq:exp}
        &\norm{\frac{1}{\sqrt{2\pi q}}\int_{-\infty}^\infty
        \big\{\exp{\i(\tilde{H}_L+\tilde{H}_B+\tilde{H}_R
        )t}\notag
        -\exp{\i(M_L+M_B+M_R)t}\big\}\exp{-\frac{t^2}{2q}}\d t}{\Lnorm}\notag\\
        &\leq 3C_1J^2\max\set{q^{1/2},\, q^{3/2}}\exp{-c_1l}
        \frac{1}{\sqrt{2\pi q}}\int_{-\infty}^\infty |t|
        \exp{-\frac{t^2}{2q}}\d t\\
        &=3(2\pi)^{-1/2}C_1J^2
        \max\set{q,\, q^{2}}\exp{-c_1l}\notag,
    \end{align}
    where the inequality can be shown using
    \cite[Chapter 9, Theorem 2.12, Equation (2.22)]{Kato}.

    Next, for the exponential term we can write
    \begin{align*}
        &\exp{\i(M_L+M_B+M_R)t}
        =\exp{\i(M_L+M_B+M_R)t}\exp{-\i(M_L+M_R)t}
        \exp{\i(M_L+M_R)t},
    \end{align*}
    and we define $U(t):=\exp{\i(M_L+M_B+M_R)t}\exp{-\i(M_L+M_R)t}$.
    In $U(t)$ the term in the exponent commutes for different $t$, since
    the time-dependence is a simple multiplication by $t$. We thus
    compute
    \begin{align*}
        &\frac{\d}{\d t}U(t)
        =\exp{\i(M_L+M_B+M_R)t}\i(M_L+M_B+M_R)
        \exp{-\i(M_L+M_R)t}\\
        &-
        \exp{\i(M_L+M_B+M_R)t}\i(M_L+M_R)
        \exp{-\i(M_L+M_R)t}\\
        &=\exp{\i(M_L+M_B+M_R)t}\i M_B\exp{-\i(M_L+M_R)t}\\
        &=\exp{\i(M_L+M_B+M_R)t}\exp{-\i(M_L+M_R)t}
        \exp{\i(M_L+M_R)t}\i M_B\exp{-\i(M_L+M_R)t}\\
        &=U(t)\exp{\i(M_L+M_R)t}\i M_B\exp{-\i(M_L+M_R)t}
    \end{align*}
    and $U(0)=\id$. We abbreviate
    \begin{align*}
        M_B(t):=\exp{\i(M_L+M_R)t}M_B\exp{-\i(M_L+M_R)t}.
    \end{align*}
    Due to the simple form of $\i M_B(t)$, the solution to this initial
    value problem exists, is unique and is
    given by the (negative) time-ordered exponential of
    $\i M_B(t)$, see \cite[Chapter X.12]{MathPhys2}
    (interaction representation) or \cite{Unbounded} for
    ordered exponentials of more general unbounded time dependent
    Hamiltonians.

    Thus, our approximation so far is
    \begin{align*}
        \bar{\rho}^q=\frac{1}{\sqrt{2\pi q}}\int_{-\infty}^\infty
        \adj{\oe{A}{0}{t}}\exp{\i(M_L+M_R)t}\exp{-\frac{t^2}{2q}}\d t.
    \end{align*}
    By multiplying $O_L$, $O_R$ from the right we obtain
    \begin{align*}
        &\norm{\bar{\rho}^qO_LO_R-\rho^0}{\Lnorm}
        =\norm{(\bar{\rho}^q-\rho^0+\rho^0)O_LO_R-\rho^0}{\Lnorm}
        =\norm{(\bar{\rho}^q-\rho^0)O_LO_R+\rho^0O_LO_R-\rho^0}{\Lnorm}
        \\
        &\leq
        \norm{\bar{\rho}^q-\rho^0}{\Lnorm}+
        \norm{\rho^0O_LO_R-\rho^0}{\Lnorm}.
    \end{align*}
    For the latter term we use Lemma \ref{lemma:ol}, set
    $q=c_1\frac{2l}{(\Delta E)^2}$ and obtain
    \begin{align*}
        &\norm{\rho^0O_LO_R-\rho^0}{\Lnorm}=
        \norm{\rho^0
        \left\{(O_L-\id+\id)+(O_R-\id+\id)-\id\right\}
        }{\Lnorm}\\
        &=\norm{\rho^0(O_L-\id)(O_R-\id)+\rho^0(O_L-\id)+
        \rho^0(O_R-\id)
        }{\Lnorm}\\
        &\leq
        3\norm{(O_L-\id)\rho^0}{\Lnorm}+
        \norm{(O_R-\id)\rho^0}{\Lnorm}
        \leq 4\exp{-c_1l/2}.
    \end{align*}
    since
    $\norm{\rho^0}{\Lnorm}=
    \norm{O_L}{\Lnorm}=1$ and $\rho^0$, $O_L$ and $O_R$ are self-adjoint.
    Thus, overall
    \begin{align*}
        \norm{\bar{\rho}^qO_LO_R-\rho^0}{\Lnorm}&\leq
        3(2\pi)^{-1/2}C_1J^2\max\set{q,\, q^2}\exp{-c_1l}
        +4\exp{-c_1l/2}.
    \end{align*}

    Finally, by definition, $O_L$ and $O_R$ project onto the spectral subspaces
    of $M_L$ and $M_R$ corresponding to small eigenvalues. I.e.,
    \begin{align*}
        \norm{\left[\exp{\i(M_L+M_R)t}-\id\right]O_LO_R
        }{\Lnorm}\leq 2|t|C_1J^2\max\set{q^{1/2},\,q^{3/2}}
        \exp{-c_1l/2},
    \end{align*}

    hence
    \begin{align*}
        \norm{\bar{\rho}^qO_LO_R-\tilde{\rho}^q}{\Lnorm}
        &\leq\bigg\|\frac{1}{\sqrt{2\pi q}}\int_{-\infty}^\infty
        \adj{\oe{A}{0}{t}}\times\\
        &\exp{-\frac{t^2}{2q}}\left(\exp{\i(M_L+M_R)t}O_LO_R-O_LO_R
        \right)\d t\bigg\|_{\Lnorm}\\
        &\leq 2C_1J^2\max\set{q^{1/2},\, q^{3/2}}\exp{-c_1l/2}
        \frac{1}{\sqrt{2\pi q}}\int_{-\infty}^\infty |t|\exp{-\frac{t^2}{2q}}\d t\\
        &=
        C_1J^2(2\pi)^{-1/2}\max\set{q,\,q^2}\exp{-c_1l/2}.
    \end{align*}

    Overall we obtain
    \begin{align*}
        \norm{\tilde{\rho}^q-\rho^0}{\Lnorm}&\leq
        \norm{\bar{\rho}^qO_LO_R-\rho^0}{\Lnorm}+
        \norm{\bar{\rho}^qO_LO_R-\tilde{\rho}^q}{\Lnorm}\\
        &\leq C_2J^2\max\set{q,\, q^2}\left(
        \exp{-c_1l/2}+\exp{-c_1l}\right),
    \end{align*}
    for an appropriately chosen constant $C_2>0$. This completes the proof.
\end{proof}

It remains to show how we can obtain a local operator $O_B$, maintaining the
same approximation order. This follows once more from a Lieb-Robinson bound.

\begin{lemma}\label{lemma:ob}
    Consider the operator
    \begin{align*}
        \OB:=\frac{1}{\sqrt{2\pi q}}\int_{-\infty}^\infty
        \adj{\oe{A}{0}{t}}\exp{-\frac{t^2}{2q}}\d t,
    \end{align*}
    with $A(t)$ as above
    \begin{align*}
        A(t):=\exp{\i(M_L+M_R)t}\i M_B\exp{-\i(M_L+M_R)t}.
    \end{align*}
    Then, there exists a local bounded operator $O_B$
    supported on $\X_{j-3l-2,j+3l+3}$, with $\norm{O_B}{\Lnorm}\leq 1$
    such that
    \begin{align*}
        \norm{\OB-O_B}{\Lnorm}\leq
        C_3J^2\max\set{q^{1/2},\, q}\exp{-c_3l},
    \end{align*}
    for some constants $C_3>0$, $c_3>0$.
\end{lemma}

\begin{proof}
    We use the same trick as in Lemma \ref{lemma:ml} to show that the non-local
    part of $M_B(t)$ is bounded.
    \begin{align*}
        M_B(t)&:=\exp{\i(M_L+M_R)t}M_B\exp{\i(M_L+M_R)t},\\
        \frac{\d}{\d t}M_B(t)&=\exp{\i(M_L+M_R)t}\comm{M_L+M_R}{M_B}
        \exp{-\i(M_L+M_R)t},\\
        \comm{M_L+M_R}{M_B}&=\comm{H_L}{H_B}+\comm{H_L}{\diff_B}+
        \comm{\diff_L}{H_B}+\comm{\diff_L}{\diff_B}+
        \comm{H_R}{H_B}\\
        &+\comm{H_R}{\diff_B}+\comm{\diff_R}{H_B}
        +\comm{\diff_R}{\diff_B}.
    \end{align*}
    Since $\diff_L$ and $\diff_R$ are supported on a superset of the supports
    of $\comm{H}{H_L}$ and $\comm{H}{H_R}$, we obtain
    \begin{align*}
        \supp{\comm{M_L+M_R}{M_B}}
        \subset\supp{\diff_L}\cup\supp{\diff_R}=\supp{M_B}.
    \end{align*}
    Thus, as in Lemma \ref{lemma:ml},
    we can approximate $M_B(t)$ by
    a local operator $\tilde{M}_B(t)$ supported on $\X_{j-3l-2,j+3l+3}$
    such that
    \begin{align*}
        \norm{M_B(t)-\tilde{M}_B(t)}{\Lnorm}\leq
        CJ^2|t|\exp{-a\{l-v|t|\}}.
    \end{align*}
    We utilize this estimate in a similar way as in \eqref{eq:exp}. To this end,
    we write the time-ordered exponential as a product integral
    and use a
    step function approximation to the integral
    (see \cite[Chapter 3.6]{ProductInt}), namely
    \begin{align*}
        \adj{\oe{A}{0}{t}}=\prod_{0}^t\exp{A(\tau)}\d\tau=
        \lim_{N\rightarrow\infty}\left(
        \exp{A(t_N)\Delta t}\cdots\exp{A(t_0)\Delta t}\right),
    \end{align*}
    where $t_i=i\Delta t$, $\Delta t=t/N$ and the convergence is meant in
    the strong sense. This is possible due to the simple form of $A(t)$,
    i.e., $\exp{A(t)}$ is bounded with norm 1.

    For $N=1$, we obtain as in \eqref{eq:exp}
    \begin{align*}
        \norm{\exp{\i M_B(t)\Delta t}-\exp{\i\tilde{M}_B(t)\Delta t}
        }{\Lnorm}\leq |\Delta t|\norm{M_B(t)-\tilde{M}_B(t)
        }{\Lnorm},
    \end{align*}
    with $|\Delta t|=|t|$. For $N-1\rightarrow N$, by induction
    \begin{align*}
        &\norm{\exp{\i M_B(t_N)\Delta t}\cdots\exp{\i M_B(t_0)\Delta t}-
        \exp{\i\tilde{M}_B(t_N)\Delta t}\cdots\exp{\i\tilde{M}_B(t_0)\Delta t}
        }{\Lnorm}\\
        &\leq
        \norm{\left[\exp{\i M_B(t_N)\Delta t}-
        \exp{\i\tilde{M}_B(t_N)\Delta t}\right]
        \exp{\i M_Bt_{N-1}\Delta t}
        \cdots\exp{\i M_B(t_0)\Delta t}}{\Lnorm}\\
        &+
        \bigg\|\exp{\i\tilde{M}_B(t_N)\Delta t}
        \times\\
        &\left[\exp{\i M_B(t_{N-1})\Delta t}\cdots\exp{\i M_B(t_0)\Delta t}-
        \exp{\i\tilde{M}_B(t_{N-1})
        \Delta t}\cdots\exp{\i\tilde{M}_B(t_0)\Delta t}\right]\bigg\|_{\Lnorm}\\
        &\leq|\Delta t|N \norm{M_B(t)-\tilde{M}_B(t)}{\Lnorm},
    \end{align*}
    where by definition $|\Delta t|N=|t|$. Thus, we can estimate for the ordered
    exponential
    \begin{align*}
        &\norm{\adj{\oe{\i M_B}{0}{t}}-\adj{\oe{\i\tilde{M}_B}{0}{t}}}{\Lnorm}\\
        &\leq |t|\norm{M_B(t)-\tilde{M}_B(t)}{\Lnorm}
        \leq CJ^2t^2\exp{-a\{l/3-v|t|\}}.
    \end{align*}
    We define the local operator $O_B$ as
    \begin{align*}
        O_B:=\frac{1}{\sqrt{2\pi q}}\int_{-\infty}^\infty
        \adj{\oe{\i\tilde{M}_B}{0}{t}}\exp{-\frac{t^2}{2q}}\d t.
    \end{align*}
    Estimating as in \eqref{eq:trunc} for $T=\frac{l}{6v}$, we obtain
    \begin{align*}
        &\norm{\OB-O_B}{\Lnorm}\leq
        \norm{\int_{|t|\leq T\ldots}+\int_{|t|>T}\ldots}{\Lnorm}
        \leq
        CJ^2\exp{-al/3}\\
        &\times\left(\exp{avT}\frac{1}{\sqrt{2\pi q}}\int_{|t|\leq T}
        t^2\exp{-\frac{t^2}{2q}}\d t+
        \frac{1}{\sqrt{2\pi q}}\int_{|t|>T}t^2\exp{av|t|-\frac{t^2}{2q}}\d t
        \right).
    \end{align*}
    In analogy to \eqref{eq:trunc}, for the first term we obtain
    \begin{align*}
        \frac{1}{\sqrt{2\pi q}}\int_{|t|\leq T}
        t^2\exp{-\frac{t^2}{2q}}\d t\leq \frac{q}{2}.
    \end{align*}
    For the second term, applying integration by parts, we obtain
    \begin{align*}
        &\frac{1}{\sqrt{2\pi q}}\int_{|t|>T}t^2\exp{av|t|-\frac{t^2}{2q}}\d t=
        \sqrt{\frac{2}{\pi q}}\int_{T}^\infty
        t^2\exp{avt-\frac{t^2}{2q}}\d t\\
        &=\sqrt{\frac{2}{\pi q}}\bigg(
        Tq\exp{avT-\frac{T^2}{2q}}+q\int_{T}^\infty \exp{avt-\frac{t^2}{2q}}
        dt
        +avq\int_{T}^\infty t\exp{avt-\frac{t^2}{2q}}\d t\bigg).
    \end{align*}
    And hence
    \begin{align*}
        &\frac{1}{\sqrt{2\pi q}}\int_{|t|>T}t^2\exp{av|t|-\frac{t^2}{2q}}\d t
        \leq\exp{avt-\frac{T^2}{2q}}\frac{\sqrt{2}}{\sqrt{\pi}|1-avq|}
        \sqrt{q}\left(\sqrt{\frac{q\pi}{2}+T}\right).
    \end{align*}
    The final estimate is thus
    \begin{align*}
        \norm{\OB-O_B}{\Lnorm}\leq
        C_3J^2\max\set{q^{1/2},\, q}\exp{-3c_3l},
    \end{align*}
    for appropriate constants $C_3>0$, $c_3>0$. This completes the proof.
\end{proof}

We are now ready to state the main result of this subsection.
\begin{theorem}\label{thm:obolor}
    Under Assumption \ref{ass:op}, there exist local, bounded and
    self-adjoint operators $O_L=O_L(j,\,l)$,
    $O_B=O_B(j,\,l)$,
    $O_R=O_R(j,\,l)$ with norms bounded by 1,
    such that for some constants $C_4>0$, $c_4>0$
    \begin{align}\label{eq:finallemma}
        \norm{O_BO_LO_R-\rho^0}{\Lnorm}\leq
        C_4J^2\exp{-c_4l}.
    \end{align}
    The respective supports are $\X_{1,j}$,
    $\X_{j-3l-2,j+3l+3}$ and $\X_{j+1,d}$. The operator $O_B$ can be chosen
    w.l.o.g.\ to be positive.
\end{theorem}

\begin{proof}
    The operators $O_L$ and $O_R$ were defined in Lemma \ref{lemma:ol} and
    their properties follow therefrom. The operator $O_B$ was defined in Lemma
    \ref{lemma:ob}. W.l.o.g.\ we can assume it is positive, otherwise
    the same arguments as in \cite[Lemma 4]{Hastings} apply.

    By Lemmas \ref{lemma:tpq}, \ref{lemma:ob} and since
    $\norm{O_LO_R}{\Lnorm}\leq 1$,
    we obtain an error bound with
    asymptotic dependence on $l$ of the form $l^2\exp{-cl}$. Hence, we
    can pick constants $C_4>0$, $c_4>0$ to satisfy \eqref{eq:finallemma}.
    This completes the proof.
\end{proof}

\subsection{Relative Entropy Bounds}
The key idea for the proof of the area law is that we can bound the relative
entropy from below by a non-vanishing term that grows with $l$. The
precise asymptotics of this lower bound will imply a constant upper bound
on the value of entropy.

\begin{lemma}\label{lemma:relent}
    Suppose Assumption \ref{ass:op} holds. Then, for
    \begin{align*}
        \E_B:=\tr{O_B(\rho^0_{1,j}\otimes\rho^0_{j+1,d})},
    \end{align*}
    and
    $\varepsilon(l):=C_4J^2\exp{-c_4l}$,
    we have the lower bound
    \begin{align}\label{eq:relent}
        &\s{\rho^0_{j-l-2,j}}+\s{\rho^0_{j+1,j+l+3}}-\s{\rho^0_{j-l-2,j+l+3}}\notag\\
        &\geq
        (1-2\varepsilon(l))\log_2\left[\frac{1-2\varepsilon(l)}{\E_B}\right]
        +2\varepsilon(l)\log_2\left[\frac{2\varepsilon(l)}{1-\E_B}\right].
    \end{align}
\end{lemma}

\begin{proof}
    The estimated quantity is commonly referred to as \emph{quantum mutual
    information}.
    We briefly show that it is equal to a specific expression of
    relative entropy. By definition of relative entropy
    \begin{align*}
        \s{\rho_{AB}||\rho_A\otimes\rho_B}=-\s{\rho_{AB}}-\tr{\rho_{AB}
        \log_2(\rho_A\otimes\rho_B)}.
    \end{align*}
    For the latter term we compute
    \begin{align*}
        &-\tr{\rho_{AB}\log_2[(\rho_A\otimes\id_B)(\id_A\otimes\rho_B)]}
        =-\tr{\rho_{AB}\left[\log_2(\rho_A\otimes\id_B)+
        \log_2(\id_A\otimes\rho_B)\right]}\\
        &=-\tr{\ptr{B}{\rho_{AB}}\log_2(\rho_A)}-\tr{\ptr{A}{\rho_{AB}}\log_2(\rho_B)}
        =\s{\rho_A}+\s{\rho_B}.
    \end{align*}

    From Theorem \ref{thm:obolor} it follows
    \begin{align*}
        \tr{\rho^0O_BO_LO_R}&=\tr{\rho^0(\rho^0-\rho^0+O_BO_LO_R)}]
        =\tr{(\rho^0)^2}+\tr{\rho^0(O_BO_LO_R-\rho^0)}\notag
        \\
        &\geq 1-\tr{|\rho^0|}\norm{O_BO_LO_R\rho^0-\rho^0}{\Lnorm}
        \geq 1-\varepsilon(l).
    \end{align*}
    Thus, applying the Cauchy-Schwarz inequality we can estimate
    \begin{align}\label{eq:lowerb}
        \tr{\rho^0O_B}=\tr{\rho_{j-l-2,j+l+3}O_B}\geq 1-2\varepsilon(l).
    \end{align}

    Next, define the map (quantum channel) $E:\Tr{\X}\rightarrow\Tr{\C^2}$
    by
    \begin{align}\label{eq:thanks2}
        E(\rho):=\tr{\rho O_B}\inp{\cdot}{(1, 0)}_{\C^2}+\tr{\rho(\id-O_B)}
        \inp{\cdot}{(0, 1)}_{\C^2}.
    \end{align}
    By Theorem \ref{thm:obolor}, $O_B$ is a bounded, positive
    operator with $\norm{O_B}{\Lnorm}\leq 1$. Thus, one easily checks that
    $E$ is a positive trace preserving map. By
    \cite[Theorem 1]{Monotone}
    the relative entropy is monotone under $E$ and we get
    \begin{align*}
        &\s{\rho^0_{j-l-2,j}}+\s{\rho^0_{j+1,j+l+3}}-\s{\rho^0_{j-l-2,j+l+3}}
        =\s{\rho^0_{j-l-2,j+l+3}||\rho^0_{j-l-2,j}\otimes\rho^0_{j+1,j+l+3}}\\
        &\geq
        \s{E(\rho^0_{j-l-2,j+l+3})||E(\rho^0_{j-l-2,j}\otimes\rho^0_{j+1,j+l+3})}
        =\tr{\rho^0_{j-l-2,j+l+3}O_B}\log_2[\tr{\rho^0_{j-l-2,j+l+3}O_B}]\\
        &+
        (1-\tr{\rho^0_{j-l-2,j+l+3}O_B})\log_2[1-\tr{\rho^0_{j-l-2,j+l+3}O_B}]\\
        &-\tr{\rho^0_{j-l-2,j+l+3}O_B}
        \log_2[\tr{\rho^0_{j-l-2,j}\otimes\rho^0_{j+1,j+l+3}O_B}]\\
        &-(1-\tr{\rho^0_{j-l-2,j+l+3}O_B})
        \log_2[1-\tr{\rho^0_{j-l-2,j}\otimes\rho^0_{j+1,j+l+3}O_B}]
        =\tr{\rho^0_{j-l-2,j+l+3}O_B}\\
        &\times\left\{\log_2[\tr{\rho^0_{j-l-2,j+l+3}O_B}]-
        \log_2[\tr{\rho^0_{j-l-2,j}\otimes\rho^0_{j+1,j+l+3}O_B}]\right\}
        +(1-\tr{\rho^0_{j-l-2,j+l+3}O_B})\\
        &\times\left\{\log_2[1-\tr{\rho^0_{j-l-2,j+l+3}O_B}]-
        \log_2[1-\tr{\rho^0_{j-l-2,j}\otimes\rho^0_{j+1,j+l+3}O_B}]\right\}\\
        &\overset{(*)}{\geq}
        (1-2\varepsilon(l))\log_2\left[\frac{1-2\varepsilon(l)}{\E_B}\right]
        +2\varepsilon(l)\log_2\left[\frac{2\varepsilon(l)}{1-\E_B}\right]
    \end{align*}
    where $(*)$ is due to \eqref{eq:lowerb}
    and since the first term is positive while the
    second is negative.
    This completes the proof.
\end{proof}

We need to replace $\E_B$ by an expectation value that is independent of
the approximation operator $O_B$.

\begin{lemma}
    For $\E:=\tr{\rho^0[\rho^0_{1,j}\otimes\rho^0_{j+1,d}]}$
    we have the bound
    \begin{align}\label{eq:pb}
        \E_B\leq\frac{\E-\sqrt{2\E_B\varepsilon(l)}+2\varepsilon(l)}{1-2\varepsilon(l)}.
    \end{align}
\end{lemma}

\begin{proof}
    Define $\E_{LR}:=\tr{O_LO_R(\rho^0_{1,j}\otimes\rho^0_{j+1,d})}$. As in
    \eqref{eq:lowerb}, $\E_{LR}\geq 1-2\varepsilon(l)$. Applying
    the same arguments as in \cite[Equation (24)]{Hastings}, i.e.,
    by applying Cauchy-Schwarz to the co-variance of operators and
    since $\E_B\leq 1$, we obtain
    \begin{align*}
        \E&\geq \tr{O_BO_LO_R(\rho^0_{1,j}\otimes\rho^0_{j+1,d})}-\varepsilon(l)
        \geq \E_B\E_{LR}-\sqrt{\E_B-\E_B^2}\sqrt{\E_{LR}-\E_{LR}^2}-\varepsilon(l)\\
        &\geq \E_B(1-2\varepsilon(l))-\sqrt{\E_B}\sqrt{2\varepsilon(l)}-
        \varepsilon(l).
    \end{align*}
    Hence,
    \begin{align*}
        \E_B\leq\frac{\E-\sqrt{2\E_B\varepsilon(l)}+2\varepsilon(l)}{1-2
        \varepsilon(l)}.
    \end{align*}
    This completes the proof.
\end{proof}

We want to use Lemma \ref{lemma:relent} to estimate entropy asymptotics
w.r.t.\ the length of a chain $l$. To this end, we need to quantify
the worst possible entropy scaling. Of course, to avoid a tautology,
this estimate has to include the worst case of exponential scaling in ranks,
i.e., linearly growing entropy.

In Section \ref{sec:entconv},
we discussed assumptions
under which the entropy is finite for any subsystem.
However, unlike in the finite dimensional case, we do
not know exactly how the entropy bounds differ from site to site.
Since a detailed investigation of this goes beyond the scope of this work, for now
we require the following assumption.

\begin{assumption}\label{ass:smax}
    We assume that there exists a constant $S_{\max}>0$ such that
    $\s{\rho^0_{j, j}}\leq S_{\max}$ for any $1\leq j\leq d$, i.e.,
    the single site entropies are bounded by $S_{\max}$. Moreover,
    we assume for any $1\leq j\leq d$ and $0\leq l\leq d-j$ that we have
    $\s{\rho^0_{j,j+l}}\leq (1+l)S_{\max}$, i.e., entropy grows at most
    linearly in $S_{\max}$
    (which still includes exponential scaling in ranks for
    a given approximation accuracy).
\end{assumption}

\begin{lemma}
    Let
    \begin{align*}
        S_l:=\max\set{\s{\rho^0_{j,j+l-1}}:\,1\leq j\leq d,\;
        [j,j+l-1]\subset[1, d]}.
    \end{align*}
    Then, under Assumptions \ref{ass:op} and \ref{ass:smax},
    there exists a constant $C_5>0$ such that
    \begin{align}\label{eq:sl}
        S_l\leq C_5(S_{\max}+1)l-\log_2(l)\log
        \left[\frac{1}{\E+2\varepsilon(l)}\right].
    \end{align}
\end{lemma}

\begin{proof}
    The 2nd term in \eqref{eq:relent} can be neglected since it vanishes
    rapidly for large $l$. For the first term we use \eqref{eq:pb} to obtain
    \begin{align*}
        &\s{\rho^0_{j-l-2,j}}+\s{\rho^0_{j+1,j+l+3}}-\s{\rho^0_{j-l-2,j+l+3}}\\
        &\geq
        (1-2\varepsilon(l))\log_2\left[\frac{1}{\E+2\varepsilon(l)}\right]
        -C_5.
    \end{align*}
    For an appropriately chosen $C_5>0$. Thus, we can estimate
    \begin{align*}
        S_{2l}\leq 2S_l+C_5-
        (1-2\varepsilon(l))\log_2\left[\frac{1}{\E+2\varepsilon(l)}\right].
    \end{align*}
    Now we simply iterate this inequality as in
    \cite[Equation (10)]{Hastings} and use Assumption \ref{ass:smax}
    to get
    \begin{align*}
        S_l\leq lS_{\max}+lC_5+\sum_{k=0}^\infty
        2\varepsilon(2^k)\log\left[\frac{1}{\E+2\varepsilon(2^k)}\right]
        -\log_2(l)\log_2\left[\frac{1}{\E+2\varepsilon(l)}\right].
    \end{align*}
    The log term in the series diverges at most linearly, thus the series
    can be bounded by a constant. Adjusting $C_5>0$ we obtain the final
    statement
    \begin{align*}
        S_l\leq C_5(S_{\max}+1)l-\log_2(l)\log
        \left[\frac{1}{\E+2\varepsilon(l)}\right].
    \end{align*}
    This completes the proof.
\end{proof}

\subsection{Expectation Value Bounds}
Equation \eqref{eq:sl} is the key for the final argument. This in turn depends
on a bound for \[\E=\tr{\rho^0[\rho^0_{1,j}\otimes\rho^0_{j+1,d}]}.\]
In order to derive
such a bound, we want to use the technique from \cite[Lemma 2]{Hastings}
and first replace $\rho^0$ by the approximation from Theorem \ref{thm:obolor},
$\rho^0\approx O_BO_LO_R$.

To this end, consider the mapping
\begin{align*}
    T_m=O_B(m)O_L(m)O_R(m)[\rho^0_{1,j}\otimes\rho^0_{j+1,d}]O_R(m)O_L(m)O_B(m),
\end{align*}
where we used $m$ to indicate the support length.
It is
trace class since $\rho^0_{1,j}\otimes\rho^0_{j+1,d}\in\State{\X}$ and the trace
class is a two sided ideal in $\LL{\X}$. Thus, we can apply the Schmidt decomposition to
$T_m$ w.r.t.\ the bipartite cut $\X=\X_{1,j}\otimes\X_{j+1,d}$.

Our bound requires some knowledge about the low-rank approximation properties
of $T_m$, i.e., convergence w.r.t.\ the Schmidt rank $r$.
Note that even though
$O_LO_R[\rho^0_{1,j}\otimes\rho^0_{j+1,d}]O_LO_R$ has rank 1, the operator
$T_m$ may have infinite rank. The operator $O_B$ is a unitary transformation
with local support. However, beyond this, it is not clear how exactly $O_B$
influences low-rank approximability. Since a deeper investigation of this
is rather intricate, for this work we focus on three cases. These cases
are prototypical for known approximation rates based on assumptions of
smoothness for the operator kernels, i.e., in our case smoothness of
eigenfunctions of $H$ (see \cite{Schneider}). Note that in all three cases
ranks scale exponentially for a given approximation accuracy.

\begin{assumption}\label{ass:tm}
    Define the normalized map
    \begin{align*}
        T_m&:=O_B(m)O_L(m)O_R(m)[\rho^0_{1,j}\otimes\rho^0_{j+1,d}]O_R(m)O_L(m)O_B(m)
        \in\Trp{\X},\\
        \rho_m&:=T_m/\tr{T_m}\in\State{\X}.
    \end{align*}
    Note that the support and thus effective dimensionality of this
    map is $2m+6$. Let $\rho_m^r$ denote the best
    rank-$r$ approximation w.r.t.\ the bipartite cut
    $\X_{1,j}\otimes\X_{j+1,d}$. For some rate $s>0$ and constant
    $C_6>0$, the three cases we consider are
    \begin{align}
        \norm{\rho_m-\rho_m^r}{\Lnorm}&\leq C_6r^{-\frac{s}{2m+6}}
        \label{eq:case1},\\
        \norm{\rho_m-\rho_m^r}{\Lnorm}&\leq C_6
        r^{-s}\log_2(r)^{s(2m+6)},\\
        \norm{\rho_m-\rho_m^r}{\Lnorm}&\leq C_6^{2m+6}r^{-s}
        \label{eq:case3}.
    \end{align}
\end{assumption}

\begin{lemma}\label{lemma:initbnd}
    Let Assumption \ref{ass:op} hold. Then, for the three rates
    in Cases \ref{ass:tm}, we get the respective bounds
    \begin{align}\label{eq:init}
        \s{\rho_{1,j}}&\leq C_7\left(1+\log_2^2\left[\frac{C_8}{\E}\right]\right),\\
        \s{\rho_{1,j}}&\leq C_7\left(1+\log_2\left[\frac{C_8}{\E}\right]
        \log_2\left\{\log_2\left[\frac{C_8}{\E}\right]\right\} \right),\\
        \s{\rho_{1,j}}&\leq C_7\left(1+\log_2\left[\frac{C_8}{\E}\right]\right)
        \label{eq:scase3},
    \end{align}
    for some constants $C_7>0$, $C_8>0$.
\end{lemma}

\begin{proof}
    We only prove the statement for the first case, the others are analogous.
    The space $\Tr{\X}$ is a two sided ideal in $\LL{\X}$.
    Moreover, we can
    interpolate the Hilbert Schmidt norm as follows
    \begin{align*}
        \normsq{\rho A}{\HSnorm}=\normsq{A\rho}{\HSnorm}=\tr{\adj{A}\adj{\rho}
        \rho A}
        \leq \norm{\rho A}{\Trnorm}\norm{\rho A}{\Lnorm}.
    \end{align*}
    And so we arrive at an estimate as in \cite[Equation (17)]{Hastings}.
    For
    \begin{align*}
        A&:=O_B(m)O_L(m)O_R(m)\in\LL{\X},\quad
        \rho:=\rho^0_{1,j}\otimes\rho^0_{j+1,d}\in\State{\X},
    \end{align*}
    we obtain
    \begin{align*}
        \tr{\rho^0T_m}&=\tr{\rho^0A\rho\adj{A}}=\inp{u_0}{A\rho\adj{A}u_0}_{\X}
        =\inp{u_0}{A\rho\adj{A}\rho^0u_0}_{\X}\\
        &=\tr{\rho^0A\rho\adj{A}\rho^0}=\normsq{\rho^0 A\sqrt{\rho}}{\HSnorm}
        =\normsq{\rho^0(A\pm \rho^0)\sqrt{\rho}}{\HSnorm}\\
        &\geq
        \left(\norm{\rho^0\sqrt{\rho}}{\HSnorm}-
        \norm{\rho^0(A-\rho^0)\sqrt{\rho}}{\HSnorm}\right)^2
        \geq
        \left(\sqrt{\E}-\varepsilon(m)\right)^2,
    \end{align*}
    where we used the identity
    \begin{align*}
        \E=\tr{\rho^0\rho}=\tr{\rho^0\rho\rho^0}=\normsq{\rho^0\sqrt{\rho}}{\HSnorm}.
    \end{align*}

    Next, note the identity
    \begin{align*}
        A-\rho^0A=A-(\rho^0)^2+(\rho^0)^2-\rho^0A=(A-\rho^0)+\rho^0(\rho^0-A).
    \end{align*}
    Moreover,
    \begin{align*}
        \tr{(\id-\rho^0)T_m(\id-\rho^0)}&=\tr{T_m-T_m\rho^0-\rho^0T_m
        +\rho^0T_m\rho^0}
        =\tr{T_m(\id-\rho^0)}
        =\tr{(\id-\rho^0)T_m}.
    \end{align*}
    And thus
    \begin{align*}
        \tr{(\id-\rho^0)T_m}=\normsq{A-\rho^0A)\sqrt{\rho}}{\HSnorm}
        \leq 4\varepsilon^2(m).
    \end{align*}
    Combining both estimates
    \begin{align*}
        \tr{\rho^0\rho_m}&\geq\frac{\tr{\rho^0T_m}}{\tr{\rho^0T_m}+4\varepsilon^2(m)}
        =1-\frac{4\varepsilon^2(m)}{\tr{\rho^0T_m}+4\varepsilon^2(m)}
        \geq 1-\frac{4\varepsilon^2(m)}{5\varepsilon^2(m)+\E-2\varepsilon(m)
        \sqrt{\E}}\\
        &=
        1-\frac{8\varepsilon^2(m)}{6\varepsilon^2(m)+
        \E+(\sqrt{\E}-2\varepsilon(m))^2}\geq
        1-\frac{8\varepsilon^2(m)}{\E}.
    \end{align*}

    We chose $r$ in \eqref{eq:case1} such that
    \begin{align*}
        \norm{\rho_m-\rho_m^r}{\Lnorm}
        \leq C_4^2J^4\exp{-2c_4m}/\E.
    \end{align*}
    Then,
    \begin{align*}
        \tr{\rho^0\rho_m^r}&=\tr{\rho^0(\rho_m^r-\rho_m+\rho_m)}\geq
        \tr{\rho^0\rho_m}-\tr{\rho^0}\norm{\rho_m-\rho_m^r}{\Lnorm}
        \geq 1-9C_4^2J^4\exp{-2c_4m}/\E.
    \end{align*}

    Let $\rho^0_r$ denote the best rank $r$ approximation to the ground state
    projection w.r.t.\ the same bipartite cut $\X_{1,j}\otimes\X_{j+1,d}$.
    Then, since clearly $\id\geq\rho_m^r\geq 0$, we obtain
    \begin{align*}
        \tr{\rho^0\rho_m^r}&=\tr{\rho^0-\rho^0(\id-\rho_m^r)}
        \leq 1-\inf_{\substack{\rankb{v_r}{j}\leq r,\\0\leq v_r\leq\id}}
        \tr{\rho^0(\id-v_r)}
        =1-\tr{\rho^0(\id-\rho^0_r)}=\sum_{k=1}^r(\sigma_k^j)^2.
    \end{align*}
    And thus
    \begin{align*}
        \sum_{k=r+1}^\infty(\sigma_k^j)^2=1-\sum_{k=1}^r(\sigma_k^j)^2
        \leq 1-\tr{\rho^0\rho_m^r}\leq 9C_4^2J^4\exp{-2c_4m}/\E.
    \end{align*}
    The above inequality readily provides enough information about the decay
    of the singular values to derive \eqref{eq:init}.

    Note that the above rank $r$ depends on $m$, i.e., $r=r(m)$.
    Let $m'$ be minimal such that
    \begin{align*}
        9C_4^2J^4\exp{-2c_4m'}/\E\leq 1.
    \end{align*}
    Then, for
    any $m>m'$,
    \begin{align}\label{eq:constr}
        \sum_{k=r(m)+1}^\infty(\sigma_k^j)^2\leq 9C_4^2J^4\exp{-2c_4m}/\E
        &=\exp{-2c_4(m-m')}(9C_4^2J^4\exp{-2c_4m'}/\E)\notag\\
        &\leq\exp{-2c_4(m-m')}.
    \end{align}
    We can now compute the maximal possible entropy satisfying this decay
    condition. As discussed in Section \ref{sec:entapp}, since entropy
    is Schur concave, it is maximized for a uniform-like distribution, subject
    to the constraint \eqref{eq:constr}.

    Thus, for the sequence $\lambda_k:=(\sigma_k^j)^2$, we maximize the von
    Neumann entropy under the conditions
    \begin{align*}
        \sum_{k=1}^\infty\lambda_k&=1,\quad
        \sum_{k=1}^{r(m'+1)}\lambda_k=1-\exp{-2c_4},\\
        \sum_{k=r(m)+1}^{r(m+1)}\lambda_k&=(1-\exp{-2c_4})\exp{-2c_4(m-m')},\quad
        m>m'.
    \end{align*}
    Computing the upper bound for the entropy
    \begin{align*}
        \s{\rho^0_{1,j}}&\leq
        -\left(1-\exp{-2c_4}\right)\log_2\left[\frac{1-\exp{-2c_4}}{r(m'+1)}
        \right]\\
        &-\left(1-\exp{-2c_4}\right)
        \sum_{n=1}^\infty\exp{-2c_4n}\log_2\left[\frac{(1-\exp{-2c_4})
        \exp{-2c_4n}}{r(m'+n+1)-r(m'+n)+1}\right]\\
        &=\log_2[r(m'+1)]-\log_2[1-\exp{-2c_4}]-\exp{-2c_4}\log_2[r(m'+1)]\\
        &+
        [1-\exp{-2c_4}]\sum_{n=1}^\infty\exp{-2c_4n}2c_4n\\
        &+
        [1-\exp{-2c_4}]\sum_{n=1}^\infty\exp{-2c_4n}
        (\log_2[r(m'+n+1)-r(m'+n)]).
    \end{align*}
    We express $r(m)$ explicitly through $m$ using Assumption \ref{ass:tm} as
    \begin{align*}
        C_6r^{-\frac{s}{2m+6}}&\leq C_4^2J^4\exp{-2c_4m}/\E\quad
        \Rightarrow\quad r(m)&\geq
        \left(\frac{C_4^2}{C_6}J^4\right)^{-\frac{2m+6}{s}}
        \E^{-\frac{2m+6}{s}}\exp{\frac{c_4}{s}[4m^2+12m]}.
    \end{align*}
    We get similar asymptotic bounds for $r(m'+n+1)-r(m'+n)$. Taking logarithms
    and
    choosing an appropriate constant $C_7>0$, we obtain
    \begin{align*}
        \s{\rho^0_{1,j}}\leq C_7(1+(m')^2-\log_2(\E)).
    \end{align*}
    Finally, from the requirement on $m'$ we compute
    \begin{align*}
        &3C_4^2J^4\exp{-2c_4m'}/\E\leq 1,\quad
        \Rightarrow\quad m'\leq \frac{1}{2c_4}\log_2\left[
        \frac{9C_4^2J^4}{\E}\right]+1.
    \end{align*}
    Thus, again choosing an appropriate constant $C_8>0$ (and possibly
    adjusting $C_7$), we get
    \begin{align*}
        \s{\rho^0_{1,j}}\leq C_7\left(1+\log_2^2\left[\frac{C_8}{\E}\right]\right).
    \end{align*}
    The proof for the other cases is analogous.
\end{proof}

\subsection{Area Law}
We are finally able to prove an area law for the ground state: the entropy
of a chain is bounded by a constant and thus does not increase with the
dimension. We are only able to show this for the third case
\eqref{eq:case3}, since for the other two the bound on $\E$ decays too
slow to apply \eqref{eq:sl} successfully. However, we conjecture that
Theorem \ref{thm:obolor} is sufficient to show approximability directly,
without going through entropy. A further investigation of this goes beyond
the scope of this work.

\begin{theorem}[Area Law]\label{thm:alaw}
    Under Assumptions \ref{ass:op}, \ref{ass:smax} and
    provided case three \eqref{eq:case3} is valid, i.e.,
    there exists a constant $0<C_{\alaw}<\infty$, independent of $j_0$ or $d$,
    such that for any $1\leq j_0\leq d$
    \begin{align*}
        \s{\rho^0_{1,j_0}}\leq C_{\alaw}.
    \end{align*}
    The constant $C_{\alaw}$
    depends on the physical properties of $H$, such as
    gap size, interaction length, interaction strength and
    Lieb-Robinson velocity.
\end{theorem}

\begin{proof}
    For any pure state of a tripartite system
    $\rho\in\State{\X=\X_A\otimes\X_B\otimes\X_C}$,
    we have by sub-additivity of entropy
    \begin{align*}
        \s{\rho_A}=\s{\rho_{BC}}\leq\s{\rho_B}+\s{\rho_C}=
        \s{\rho_B}+\s{\rho_{AB}}.
    \end{align*}
    Hence, $\s{\rho_{AB}}\geq\s{\rho_A}-\s{\rho_B}$. Moreover, partial traces
    can be ``chained'', in the sense that
    $\ptr{AB}{\cdot}=\ptr{A}{\cdot}\circ\ptr{B}{\cdot}=\ptr{B}{\cdot}\circ\ptr{A}{\cdot}$.
    To see this, let $T_A\in\LL{\X_A}$ and $T_{AB}\in\LL{\X_{AB}}$.
    Then, by definition of the partial trace, for $\rho\in\State{\X}$
    \begin{align*}
        \tr{\ptr{BC}{\rho}T_A}&=\tr{\rho T_A\otimes \id_{BC}},\\
        \tr{\ptr{C}{\rho}T_{AB}}&=\tr{\rho T_{AB}\otimes\id_C},\\
        \tr{\ptr{B}{\ptr{C}{\rho}}T_{A}}&=\tr{\ptr{C}{\rho}T_A\otimes \id_B}
        =\tr{\rho T_A\otimes\id_B\otimes\id_C}
        =\tr{\rho T_A\otimes\id_{BC}}.
    \end{align*}
    Thus, for $k>j_0$ and by applying Assumption \ref{ass:smax},
    \begin{align*}
        \s{\rho^0_{1,k}}\geq\s{\rho^0_{1,j_0}}-\s{\rho^0_{j,k}}
        \geq\s{\rho^0_{1,j_0}}-(k-j_0)S_{\max}.
    \end{align*}
    Hence, for $j_0\leq k\leq j_0+l_0$ and some $\theta\in(0,1)$, where
    $l_0\leq\frac{1-\theta}{S_{\max}}\s{\rho^0_{1,j_0}}$,
    $l_0+1\geq\frac{1-\theta}{S_{\max}}\s{\rho^0_{1,j_0}}$, we have
    $\s{\rho^0_{1,k}}\geq \theta\s{\rho^0_{1,j_0}}=:S_{\cut}$.

    Applying \eqref{eq:scase3},
    \begin{align*}
        \E\leq C_8\exp{-\frac{S_{\cut}}{C_7}+1},
    \end{align*}
    where the inequality is valid for any $j\in[j_0,j_0+l_0]$.
    With this bound, we can pick the constants $\theta$ and $S_{\max}$ such
    that for $l\leq l_0$, either $\E\leq \varepsilon(l)$ or $S_{\cut}$ is
    bounded by a constant. The latter automatically bounds
    $\s{\rho^0_{1,j_0}}$,
    so we only consider the former.
    Then, equation \eqref{eq:sl}
    becomes
    \begin{align*}
        S_l\leq C_5(S_{\max}+1)l+\log_2(l)\log
        \left[3\varepsilon(l)\right],
    \end{align*}
    for any $0\leq l\leq l_0$. The positive term scales linearly in $l$
    while the negative term scales log-linearly in $l$. Thus, since entropy
    is always positive, $l_0$ can not be arbitrarily large and
    therefore, by
    definition of $l_0$, $\s{\rho^0_{1,j_0}}$ must be bounded.
    This completes the proof.
\end{proof}

\section{Summary}\label{sec:summ}
In this work, we investigated the properties of a Hamiltonian that allow for
low-rank approximability.
To this end, we have exploited the vast knowledge and experience
available in the literature on quantum entanglement. In Section
\ref{sec:entropy},
we have shown that approximability
is linked to entropy scaling. Though this
characterization is not complete, it is rather extensive for 1D systems
(TT format).
We also discussed the issue of entropy continuity in infinite
dimensions and considered a common model setting.

In Section \ref{sec:alaw}, we have shown how local interactions
in a Hamiltonian
lead to eigenfunctions, whose projectors can be well
approximated by local operators. We have also shown that, under further
assumptions on the approximand, this implies an area law for
the von Neumann entropy.

While we demonstrated the essential mathematical techniques and benefits
of the entropy approach,
many issues remain. In the following we discuss some of these questions.
Since the proofs presented in this chapter
are rather lengthy and technical, we begin by reviewing
the main steps.

\subsection*{Key Ingredients}
The starting point to derive an entropy estimate is the lower bound on
relative entropy in \eqref{eq:relent}.
The fact that we could derive \eqref{eq:relent}
relies on \eqref{eq:finallemma}, i.e., the eigenfunctiion projection
$\rho^0$ is essentially a product
of three local operators. This, in turn, is directly implied by the local
structure of $H$, see Assumptions \ref{ass:op}. Note that it is not essential
that we considered approximating $\rho^0$: any part of the spectrum of $H$ would
do, as long as it is separated from the rest of the spectrum.

The derivation of \eqref{eq:finallemma}
essentially relies on the spectral decomposition of $H$ and the
fact that we can express parts of $H$ through commutators. At this point the
locality of $H$ comes in, since these commutators reduce to local operators
with small support. The approximating arguments rely on
Lieb-Robinson type
bounds, i.e., support/information has a finite propagation speed depending
solely on $H$.

Having derived the bound on relative entropy in \eqref{eq:relent},
the last step is to
show that this lower bound scales sufficiently fast,
such that \eqref{eq:sl}
can not be valid for values of entropy that are too large.
The scaling of the lower
bound is then determined by the expectation value
$\E=\tr{\rho^0[\rho^0_{1,j}\otimes\rho^0_{j+1,d}]}$, which was bounded in Lemma
\ref{lemma:initbnd}. Note, that since $\rho^0=\rho^0_{1,d}$ is a pure state,
it is factorized if and only if $\uu_0$ is rank one. Thus,
$\E$ can be seen as a measure of entanglement for $\rho^0$ w.r.t.\
the bipartite cut $\X=\X_{1,j}\otimes\X_{j+1,d}$.

Although all
three bounds seem to have similar asymptotic behavior, only the last
one yields an area law. The reason is that in \eqref{eq:sl} we consider
$C_1l-C_2\log_2(l)l$ and argue that this becomes negative. If we would have
$C_1l-C_2\log_2(l)l^p$ for any $p<1$, the argument fails. This delicate
balance, thus, restricts us to \eqref{eq:scase3}.

\subsection*{Assumptions on the Hamiltonian}
We have already discussed Assumption \ref{ass:op} in Remark \ref{rem:ass}.
Local interactions and a gap are
necessary. We could have considered any part of the spectrum, not
necessarily the ground state, as long as we have a gap above and
possibly below.
For non-local interactions, one could derive similar estimates,
given
sufficient decay of interaction strength for long-range interactions.

Self-adjointness is a technical assumption,
necessary for the spectral decomposition. This can be perhaps generalized.

Finite interaction strength is necessary for a finite support propagation
speed. However, since we only require the application of Lieb-Robinson bounds,
any interaction that admits such bounds would work. It is known that
this is not possible for any unbounded interaction (see \cite{SuperSonice}).
Nonetheless, some unbounded interactions
can be controlled
(see \cite{LBInfDim, LRComm})
such that the finite interaction strength can be generalized.

The weakest point is, however, the requirement made in
Assumption \ref{ass:tm}. In essence,
$O_B$ was constructed by approximating the spectrum of $H$. However, we were
not explicit in the construction such that it is not clear to us what effect
$O_B$ has on the low-rank structure of $\rho_m$.
An explicit
construction of $O_B$ may be possible due to \cite[Lemma 3.2]{NearComm}.
Although we believe that
\eqref{eq:relent} is a necessary step to bound the entropy, we are not certain
if Lemma \ref{lemma:initbnd} could be entirely avoided.

The result of
Theorem \ref{thm:obolor} is interesting in its own right as it is already a
statement about approximability and avoids many seemingly artificial
assumptions required for the final area law in Theorem \ref{thm:alaw}.
Theorem \ref{thm:obolor}
essentially states
\begin{align*}
    \rho^0\approx O_BO_LO_R,
\end{align*}
where $O_LO_R$ is rank-one, $O_B$ has an overlapping support of size
$l$ and the error bound depends exponentially on $l$ only, not on the
full dimension $d$. For the final low-rank approximation, we would have to
approximate $O_B$ with low-rank. Intuitively, since $O_B$ has support size $l$,
a low-rank approximation to $O_B$ has a complexity scaling at worst
exponentially in $l$, not $d$. Indeed, if the underlying Hilbert spaces
are finite-dimensional, this is trivially satisfied. However,
it is not clear to us how to express this correctly if the Hilbert spaces
are infinite dimensional --- hence the need for Assumption \ref{ass:tm}.

\subsection*{Electronic Schrödinger Equation}
Finally, we mention limitations concerning the well known electronic
Schrödinger
equation:
\begin{align}\label{eq:H}
    H=K+V=-\frac{1}{2}\sum_{k=1}^N\Delta_k-\sum_{k=1}^N
    \sum_{\nu=1}^M\frac{Z_\nu}{|x_k-a_\nu|}+\frac{1}{2}\sum_{k,j=1}^N
    \frac{1}{|x_k-x_j|},
\end{align}
where $Z_\nu$ and $a_\nu$ are the nuclei charges and positions, respectively,
$x_1,\ldots,x_N\in\R^3$.
Clearly, the last term in the potential is not local. Moreover,
due to the singularity in the potential, the interactions are only
\emph{relatively} bounded (see \cite[Chapter 4.1.1]{Kato}).
Thus, Assumption
\ref{ass:op} does not apply.

We may still obtain
a bound as in \eqref{eq:finallemma}, if the interactions are sufficiently small
for large $r_{kj}$, where $r_{kj}$ measures the distance between the sites
$k$ and $j$. The different electrons in \eqref{eq:H} interact
equally\footnote{By the nature of the described physical phenomenon, there is
no underlying lattice-like structure.} for any $|k-j|$ so that we do not see how to extend the results to
this case.
We refer to \cite{LongRange}
for area laws with long range interactions.
On the other hand, Lieb-Robinson bounds could be perhaps extended to
relatively bounded potentials.
We mention
\cite{LBInfDim, LRComm} for Lieb-Robinson bounds for unbounded interactions.

\subsection*{Entropy Measures}
When switching to the infinite-dimensional regime, the von Neumann entropy
is no longer necessarily continuous or finite. In fact, it is discontinuous for
`most' states in $\Tr{\X}$.

In Proposition
\ref{prop:finentropy} we have seen that entropy is finite if
the expected energy is finite and the Gibbs state exists for any inverse
temperature. For eigenstates the energy is obviously finite.
The existence of a Gibbs state is more restrictive, though certainly there are
many examples where this is true. This is the case if the ionization
threshold diverges. The spectrum is discrete in the presence
of a diverging (confinement) potential or if
the domain is bounded (e.g., infinite potential well).

However, for instance, the spectrum of
\eqref{eq:H} is not purely discrete (see \cite[Thm. 5.16]{Harry}), since
the ionization threshold remains bounded
and thus the Gibbs state does not exist in this case.
This suggests that, at the very least, other entropy measures are
worth a consideration for
PDEs. Indeed, the question of possible alternative entropy measures,
particularly for infinite dimensions, has been previously addressed.
We refer to, e.g., \cite{QuantEnt, RelentInfdim, PlenioInfDim, SqInt,
PlenioIntro, Exchange, Pure, Cf} for more details.

\subsection*{General Right-Hand-Side}
In this work we considered the low-rank structure of eigenfunctions. In general,
both in application and approximation theory, one would be interested in
general right-hand-sides. Put precisely, given that the right-hand-side
is low-rank, does the same hold for the solution?
We are only aware of one work that addressed this question for PDEs
in \cite{DahmenTS}. In \cite{AndreNNI} the authors successfully utilized area laws
for spin systems as in \cite{AradSubExp}, to derive low-rank approximability estimates
for discretized PDEs with a general right hand side.

In \cite{DahmenTS} the authors considered a Laplace-like PDE operator. In particular,
no interactions are involved. For such an operator the eigenfunctions are
rank one tensor products. The authors use an explicit representation
for the inverse operator and an exponential sum approximation for the
inverse eigenvalues. The trick is then to show that this approximation does
not significantly increase ranks, in the sense that the solution is
in a slightly worse approximation class than the right hand side.

In the spirit of \cite{AndreNNI}, one could try to extend results from
spectrum approximability to approximability of general solutions. Possibly
utilizing ideas as in \cite{DahmenTS}, i.e., the eigenfunctions are no longer
rank-one but are TT-approximable. However, this is a far from trivial task and
will probably require stronger restrictions on the operator structure.

\subsection*{Beyond the TT Format}
We analyzed approximability of 1D systems within the MPS/TT format. It is
generally known that the approximation format should fit the PDE structure
if one is to obtain good approximation results. By now an entire variety of
higher-dimensional tensor structures is available, see, e.g., \cite{TN} for
a recent overview. A natural question is thus: to what extent does the
above apply to multi-dimensional systems?

Even though multi-dimensional formats can be very successful for tailored
applications, a general theory of approximability seems elusive. Firstly,
general tensor networks
with loops are not closed (see \cite{NotClosed}).

Secondly, on one hand, the holographic principle seems robust w.r.t.\ the
dimension: area laws hold in a thermal equilibrium in any
dimension \cite{AlawNd}.
On the other hand,
area laws are insufficient to describe approximability in higher dimensions.
In, e.g., \cite{AlawCritical} multi-dimensional bosonic
systems in quantum critical states
were shown to satisfy area laws.

Thus, it seems an approximation theory of multi-dimensional systems would
require physical and mathematical ideas fundamentally different from the
ones applied in this work...

\bibliographystyle{acm}
\bibliography{literature}

\end{document}